\documentclass{article}
\usepackage[utf8]{inputenc}
\usepackage{comment}
\usepackage{bbm}
\usepackage{amsmath,amssymb}
\usepackage{graphicx,caption}
\graphicspath{{Figures/}}
\usepackage[usenames,dvipsnames]{xcolor}
\usepackage{pgfplots} 
\pgfplotsset{compat = newest}
\usepgfplotslibrary{fillbetween}
\usetikzlibrary{patterns}
\usepackage{dsfont,url}

\newtheorem{theorem}{Theorem}

\newtheorem{proposition}[theorem]{Proposition}
\newtheorem{definition}[theorem]{Definition}
\newtheorem{corollary}[theorem]{Corollary}
\newtheorem{lemma}[theorem]{Lemma}

\newtheorem{remark}[theorem]{Remark}
\newtheorem{example}[theorem]{Example}

\newenvironment{proof}[1][Proof]{\noindent\textbf{#1.} }{\ \rule{0.5em}{0.5em} \newline}

\def\qed{\hbox{${\vcenter{\vbox{		 
   \hrule height 0.4pt\hbox{\vrule width 0.5pt height 6pt
   \kern5pt\vrule width 0.5pt}\hrule height 0.4pt}}}$}}

\def\cD{\mathcal D}

\def\cF{\mathcal F}

\def\cL{\mathcal L}

\def\bE{\mathbb E}

\def\bL{\mathbb L}

\def\bP{\mathbb{P}}
\def\bQ{\mathbb Q}
\def\bR{\mathbb R}

\def\exist{\mathop{\exists}\limits}

\newcommand\myeq{\mathrel{\overset{\makebox[0pt]{\mbox{\normalfont\tiny\sffamily $d$}}}{=}}}

\newcommand{\red}[1]{\textcolor{red}{#1}}
\newcommand{\blue}[1]{\textcolor{black}{#1}}
\newcommand{\ste}[1]{\textcolor{magenta}{#1}}

\title{Gaussian Volterra processes as models of electricity markets} 

\author{Yuliya Mishura \footnote{Department of Probability, Statistics and Actuarial Mathematics, Taras Shevchenko National University of Kyiv, M\"{a}lardalen University, {\tt yuliyamishura@knu.ua}}  \and Stefania Ottaviano\footnote{Department of Mathematics ``Tullio Levi Civita”, University of Padova, {\tt stefania.ottaviano@math.unipd.it}} \and Tiziano Vargiolu\footnote{Department of Mathematics ``Tullio Levi Civita”, University of Padova, {\tt vargiolu@math.unipd.it}} }
\date{}
 
\begin{document}

\maketitle
\begin{abstract}
We introduce a non-Markovian model for electricity markets where the spot price of electricity is driven by several Gaussian Volterra processes, which can be e.g., fractional Brownian motions (fBms), Riemann-Liouville  processes or Gaussian-Volterra driven Ornstein-Uhlenbeck processes. 
 Since in energy markets the spot price is not a tradeable asset, due to the limited storage possibilities, forward contracts are considered as traded products. 
We 
 ensure necessary and sufficient conditions for the absence of arbitrage that, in this kind of market, reflects the fact that the prices of the forward contracts are (Gaussian) martingales under a risk-neutral measure. 
Moreover, we characterize the market completeness in terms of the number of forward contracts simultaneously considered and of the kernels of the Gaussian-Volterra processes.\\
\blue{We also provide a novel representation of Ornstein-Uhlenbeck (OU) processes driven by Gaussian Volterra processes. Also exploiting this result, we show analytically that, for some kinds of Gaussian-Volterra processes driving the spot prices, under conditions ensuring the absence of arbitrage, the market 
 is complete.    \\
Finally, we formulate a 
 portfolio optimization problem 
for an agent who invests in an electricity market, and we solve it explicitly in the case of CRRA utility functions. We also find closed formulas for the price of options written on forward contracts, together with the hedging strategy.}
 \\ 
\end{abstract}



\section{Introduction}

Soon after the liberalization of electricity markets in Europe, I. Simonsen in \cite{Simonsen} presented empirical evidence that electricity prices demonstrate
 stylized facts (most notably, antipersistence and self-similarity) which are well explained by fractional Brownian motion (fBm). This was confirmed in several successive empirical studies (see e.g. \cite{GioMor} and references therein). However, the wide application of fBm in energy markets, as well as  in general  financial markets, was hindered by the fact that fBm is not a semimartingale, thus (roughly speaking) it produces the possibility of arbitrage (see e.g. the wide discussions on this aspect in \cite[Chapter 7]{BHOZ} and \cite[Chapter 5]{Mishura}). Partial remedies to this could be the introduction of transaction costs, or the use of mixed-fBm models; however, both of these models are comparatively difficult to consider analytically.

The aim of this paper is to show that it is indeed possible to use fBm, and even more general Gaussian Volterra processes, in electricity markets without introducing arbitrage. We also provide an application of our model to portfolio optimization  and option pricing problem. 

In order to illustrate our approach, we start with the following toy model.   Assume that the spot price of electricity is of the form $S_t = \varphi(t) + \bar S_t$, where $\varphi$ is a deterministic seasonality function and $\bar S$ evolves according to Langevin equation as 
\begin{equation} \label{OUintro}
d\bar S_t = - \lambda \bar S_t\ dt + \sigma\ dB^H_t 
\end{equation}
under the real-world probability measure $\bP$, 
where $\lambda > 0$ is the mean-reversion speed, $\sigma > 0$ is a constant volatility and $B^H$ is a fBm with Hurst exponent $H \in (0,1)$ (the value   $H = 1/2$ corresponds to the well-known case when $B^H$ is a standard Brownian motion) with respect to a given filtration $(\cF_t)_{t \geq 0}$. This means that $B^H$ is a centered Gaussian process such that $B^H(0) = 0$ and that its covariance function is given by
$$ \bE[B^H_s B^H_t] = \frac12 (s^{2H} + t^{2H} - |t - s|^{2H} ). $$
In turn, this implies that the correlation of future increments of $B^H$ with the past trajectory is positive (i.e., $B^H$ is {\em persistent}) when $H > 1/2$ and negative (i.e., $B^H$ is {\em antipersistent}) when $H < 1/2$. 
 Moreover, for $H > 1/2$, the decay of this dependence as the time intervals grow
apart is slow, and we talk about long-range dependence (long memory). For $H < 1/2$, this decay
is fast, and referred by short-range dependence (short memory).
The Hurst index is also a measure for the roughness of the paths of
fBm: the larger the Hurst index, the smoother the paths. Other basic properties captured by $H$ are stationarity of increments. that is $(B^H_{t+h}-B_h^H)_{t \geq 0} \myeq (B_t)_{t \geq 0}$, $h>0$ and
self-similarity, that is $(a^{-H}B^H_{at})_{t \geq 0}\myeq (B_t^H)_{t \geq 0}$, $a>0$. 
It can be shown that $B^H$ is the unique centered $H$-self-similar
Gaussian process with stationary increments. We underline, moreover, that if $H\in (0,\frac{1}{2}) \cup (\frac{1}{2},1)$ the fBm is neither a Markov process nor a semimartingale (as most Gaussian-Volterra processes).
\cite{JC1,Mishura}.
These 
properties are transferred also to the process $\bar S$, and consequently to the spot price $S$.

One of the main peculiarities of electricity market is however that the spot price of electricity $S$ is not liquidly traded. What is traded instead are forwards with delivery periods.
 In our model, we focus on the instantaneous forward price $F(t,T)$, where $T \leq \bar T$ is the time to maturity and $t \in [0,T]$. 
Assuming that there exists a risk-neutral pricing measure $\bQ$ equivalent to the measure $\bP$, the forward price of electricity at time $t$ with maturity $T > t$ can be expressed as  
\begin{equation} \label{forwardintro}
F(t,T) = \bE_\bQ[S_T\ |\ \cF_t].
\end{equation}
This means that, under the pricing measure $\bQ$, the prices $F(\cdot,T)$ of traded forward prices are  automatically martingales for any maturity $T$. We take   the filtration $(\cF_t)_{t\ge 0}$ that is the filtration  generated by the   underlying Brownian motion   $W$,  for which the fBm $B^H$ can be written via   Molchan-Golosov  representation
(\cite[Thm 5.2]{NVV})  as 
\begin{equation} \label{MVNintro}
B^H_t = \int_0^t K(t,s)\ dW_s, \qquad \forall t \geq 0,  
\end{equation}
where $K$ is a deterministic Volterra-type kernel (more details are given 
 in Section \ref{GVP}). The filtrations generated by $W$ and $B^H$ coincide.  Thus,  in this case the dynamics of $F(\cdot,T)$ under $\bQ$ is given by the relation 
\begin{equation} \label{dFintro}
dF(t,T) = \bar{K}(T,t)\ dW^\bQ_t,
\end{equation}
where $W^\bQ$ is a $\bQ$-Brownian motion, obtained from $W$ via a Girsanov transformation, and $\bar{K}$ is a Volterra kernel, connected to $K$ via an exponent function appearing when we solve Langevin equation  \label{OUintro} (see Proposition \ref{prop17} for the details). 

This toy model can be generalized in several directions. We do this by assuming that the deseasonalized spot price $\bar S$ is the sum of several Gaussian factors, each one being a different Volterra process with respect to a $n$-dimensional standard Brownian motion $W$. This Volterra processes can be e.g. fBms, or Riemann-Liouville processes, or fBm-driven Ornstein-Uhlenbeck mean-reverting processes as in Equation \eqref{OUintro}, possibly with different Hurst exponents; as particular case of this we could have mixed fBm, which are sums of fractional and standard Brownian motions. By assuming that the filtration $(\cF_t)_{t\ge 0} $ is generated by $W$ and that there exists a pricing measure $\bQ$ equivalent to $\bP$ such that the forward prices are defined by Equation \eqref{forwardintro}, we prove that their dynamics follow a multidimensional version of Equation \eqref{dFintro}. Thus, starting from the spot price being a possibly non-Markovian and non-semimartingale process,  we obtain  that forward prices $F(\cdot,T)$, with any fixed maturity $T$, are Gaussian martingales. We also provide   necessary   and sufficient conditions for this market to be complete, based on the number of forward contracts simultaneously traded in the market and on the kernels of the Volterra processes.

In doing this, we follow the main stream of literature, justified also by market practice, to model the forward prices $F(t,T)$, relative to an instantaneous delivery of electricity at time $T$. However, the real contracts traded in electricity markets use instead the quantities $F(t,T_j,T_k)$, each one of these being the price at time $t$ of a so-called {\em flow forward} or {\em swap} contract (see e.g. \cite{benth2008,BPV,HinWag,PSV}), which delivers a fixed intensity of electricity over the period $[T_1,T_2]$, where $0\le t \le T_1 < T_2$. A common market practice when the delivery period $T_2 - T_1$ is small (for example, for hourly or daily forward contracts) is to model not directly the traded quantities $F(t,T_1,T_2)$, but rather to model $F(t,\tilde T)$ with $\tilde T \in [T_1,T_2]$ as a proxy (for example, $\tilde T$ being the middle point of $[T_1,T_2]$). We show in Lemma \ref{2H} that this market practice is justified also in our framework, as we prove that, if the kernel $K(T,t)$ is H\"older in $T$ with some exponent $\rho$ (feature shared by many kernels of fractional processes, see Section 5 for details), then the tracking error between the real price $F(t,T_1,T_2)$ and the proxy $F(t,\tilde T)$ tends to zero with order $(T_2 - T_1)^{2\rho}$. 
\blue{However, let us note that when the delivery period is of the order of one month, one trimester or one year, in which case the approximation of Lemma \ref{2H} can possibly be not accurate enough, we can still use the results of absence of arbitrage and completeness of our model, but take care of substituting the kernels $K(T_j,t)$, relative to ``instantaneous'' forwards introduced in Equation \eqref{S-F}, with the kernels $\bar K(t,T_j,T_k)$ in Equation \eqref{barK}, relative to the forward contracts which are actually traded in the market.}
 
\blue{We provide some examples of Gaussian Volterra processes which can be used to model electricity spot prices, such as Riemann-Liouville (RL) processes, fractional Brownian motions, 
as well as Ornstein-Uhlenbeck (OU) processes driven by them, and lastly a mixed case where we consider an OU process and an OU process driven by a fBm. We show analytically that, when these kind of processes drives the spot prices, under conditions ensuring the absence of arbitrage, market completeness is satisfied.
We also provide a novel representation of Ornstein-Uhlenbeck (OU) processes driven by Gaussian Volterra processes and, accordingly, we show how to treat fractional Brownian motions as a particular case of fractional OU processes. When this can be applied, we use this result for our analysis.} 

We apply this model to the problem of portfolio optimization in electricity markets. More precisely, we assume that an agent can trade in several forward contracts having maturities $T_1 < T_2 < \ldots < T_m$, and wants to find the optimal portfolio by maximizing the expected utility of the 
terminal wealth, thus resulting in a stochastic control problem. It is well known that, for non-Markovian processes  
classical methods based on dynamic programming and Hamilton-Jacobi-Bellman equation are not suited. Apart from the special case of linear-quadratic problems, methods which are used in this framework are based on the stochastic maximum principle or on the martingale approach (see e.g. \cite[Chapter 9]{BHOZ} for several examples). Here we choose to use the martingale approach: however, differently from \cite[Chapter 9.5]{BHOZ}, our framework allows us to be free not to use complex tools like the Wick-Ito-Skorohod integral. Instead, by using the Molchan-Golosov compact interval representation in Equation \eqref{MVNintro} and the dynamics \eqref{dFintro} implied by it, we are able to represent the optimal {wealth process} as a suitable deterministic function, depending on the utility function and on the initial capital, computed in terms of the density $\frac{d\bQ}{d\bP}$, which is of Girsanov type. Assuming that the Girsanov kernel in $\frac{d\bQ}{d\bP}$ is deterministic, we are able to provide analytically the optimal portfolio {strategy for an agent maximising a Constant Relative Risk Aversion utility}.
Then, we consider the problem of pricing for vanilla calls and puts written on a traded forward, by using a version of the Bachelier formula; in this case we also provide the hedging strategy. We give also the pricing formula for a more exotic derivative used in electricity market, called reliability option.

The paper is organized as follows. In Section 2, we first introduce  the model in its full generality, with the spot price driven by $m$ Gaussian-Volterra processes. Then, we provide sufficient and necessary conditions for the absence of arbitrage, as well as for the completeness of the market. 
\blue{In Section 3 we give some examples of Gaussian Volterra processes which can be used to model electricity spot prices, {examples which include various kinds of fractional processes}. We show analytically that, when this kind of processes drives the spot prices, under conditions ensuring the absence of arbitrage, market completeness is satisfied. We point out that in this section there are mathematical results that are interesting on their own, such as Proposition \ref{prop17} and Lemma \ref{lemmaKY}, where we provide a novel representation of Ornstein-Uhlenbeck (OU) processes driven by Gaussian Volterra processes. 
Section 4 states the optimization portfolio problem in terms of utility maximisation problem for an agent who invests in an electricity market; we present explicit solution to this problem for the case of Constant Relative Risk Aversion (CRRA) utility functions. 
Finally, in Section 5, we present the problem of pricing and hedging for classical vanilla options, i.e. calls and
puts, written on the forward contracts, as well as the pricing of the so-called Reliability Options.}


\section*{Acknowledgements}

This work was initiated while the first author was visiting the University of Padova in July and August 2022 under the framework of ``UniPD Rescue Fund''. The first  author is also supported by The Swedish Foundation for Strategic Research, grant Nr. UKR22-
0017 and by Japan Science and Technology Agency CREST, project reference number JPMJCR2115.\\ 
The second author is supported by the European Union - FSE-REACT-EU, PON Research and Innovation 2014-2020 DM1062/2021 and by the INdAM - GNAMPA Project code CUP E53C23001670001.\\
The third author 
is supported by the INdAM - GNAMPA Project code CUP E53C23001670001 and by the 
projects funded by the EuropeanUnion – NextGenerationEU under the National Recovery and Resilience Plan (NRRP), Mission 4 Component 2 Investment 1.1 - Call PRIN 2022 No. 104 of February 2, 2022 of Italian Ministry of University and Research; Project 2022BEMMLZ (subject area: PE - Physical Sciences and Engineering) ``Stochastic control and games and the role of information'', and Call PRIN 2022 PNRR No. 1409 of September 14, 2022 of Italian Ministry of
University and Research; Project P20224TM7Z (subject area: PE - Physical
Sciences and Engineering) ``Probabilistic methods for energy transition''. 
The authors wish to thank for fruitful discussions 
Fabio Antonelli,
Pietro Manzoni,
Marco Mastrogiovanni,
Wolfgang Runggaldier, Anton Yurchenko-Tytarenko, and also
\blue{two anonymous referees for their valuable comments and suggestions which helped to improve this paper.}


\section{The model}\label{Sec:theModel}

Let $(\Omega, \cF, \bP )$ be a probability space endowed with the filtration $(\cF_t )_{t \geq 0}$ generated by a $n$-dimensional standard Brownian motion $\{W_t, t \geq 0\}$,  completed with all the $\bP$-null sets. We assume that all the stochastic processes considered below are adapted to this filtration. 
 
We now fix a final horizon $\bar T   < + \infty $, and assume that the spot price of electricity at each time $t \in [0,\bar T]$ is a $n$-factor process, i.e. driven by $n$ Gaussian Volterra processes, see, e.g. \cite{MSS}, as
\begin{equation}\label{St}
S_{ t}  = \sum_{i=1}^n \int_0^{ t}  K_i(t,s) dW^i_s + \varphi(t) =  
 \int_0^{t} K(t,s) \cdot dW_s + \varphi(t),
\end{equation}
where the deterministic functions $K_i : [0,\bar T]^2 \to \bR$, $i = 1,\ldots,n$ are Volterra kernels  such that $\int_0^{t} K_i^2(t,s) ds < \infty$ for all $t \in [0,\bar T]$, $i = 1,\ldots,n$, and $\varphi \in C^0([0,\bar T])$ is a deterministic seasonality function. The integral term in the right-hand side of Equation \eqref{St} contains a scalar product  of  the $n$-dimensional Brownian motion $W$ and the $\bR^n$-valued function $K$, defined as
\begin{equation} \label{Kvec}
K(t,s) := (K_1(t,s), \ldots, K_n(t,s)) \qquad \forall s, t \in [0,\bar T].
\end{equation}
 
This allows the process $S$ to be the sum of several Gaussian components, e.g. the sum of a Brownian motion and an Ornstein-Uhlenbeck (OU) process as in \cite{LucSch}, or a variant of this model driven by fractional Brownian motions and/or fractional OU processes, possibly with different Hurst exponents, and so on. As one can see,  the toy model presented in the Introduction  corresponds to a fractional OU process.  More details about these models, together with details about the Volterra kernels suited for this, can be found in Section \ref{GVP}.
 
We assume that the spot price of electricity $S = (S_{t})_{t\in[0,\bar{T}]}$ is not traded in the market, but the traded products are instead $m$ forward contracts with maturities $0 < T_1 < \ldots < T_m \leq \bar T$, and we assume that the general 
dynamics of the $j$th forward price $F(\cdot,T_j)$ with maturity $T_j$, $j = 1,\ldots,m$, evolves as 
\begin{equation} \label{FP}
dF(t,T_j) = \mu(t,T_j)\ dt + \sigma(t,T_j) \cdot dW_t 
\end{equation}
for $t < T_j$, with the terminal condition $F(T_j,T_j) = S_{T_j}$, where $\mu$ and $\sigma$ are deterministic functions defined for $0 \leq t \leq T_j \leq \bar T$,
with values respectively in ${\mathbb R}$ and ${\mathbb R}^n$, such that $\mu(\cdot,T_j)$ is integrable and $\sigma(\cdot,T_j)$ is square integrable on $[0,T_j]$  for all $j = 1,\ldots,m$.
{To complete the picture, we assume (as usual) that also a riskless asset can be traded: the price of this asset can be assumed to be identically equal to 1 without loss of generality (see e.g. \cite[Chapter 10.3]{Bjork} or Definition 12.1.1. and subsequent comment in \cite{Oksendal} for details).}


In the following section, we will discuss the conditions for absence of arbitrage and completeness in our market. 
 In doing this, we follow closely the framework present in \cite{Oksendal}, which is especially suited to the general dynamics presented in Equation \eqref{FP}\footnote{ whereas other similar sources  are more specialized in assuming that asset prices are strictly positive, see e.g. \cite{Bjork,Pascucci}}. As it is well known, these notions are related to that of portfolio. As we already pointed out, in electricity markets, the spot price $S$ is not 
traded (and, if a proxy for that is quoted, it cannot be stored without transaction costs, that in this case means some energy loss), while forward contracts with maturities $T_j$, $j = 1,\ldots,m$ are.
Thus, we assume that an agent can build a financial portfolio by investing at time $t$ the quantity $\Delta_t^j$ in the forward contract with maturity $T_j$, $j=1, \ldots, m$.  To keep things simple, we assume that our agent is interested in trading in this market only while all these forward contracts are still traded: since each forward contract is defined only until its maturity $T_j$, the global portfolio position $\Delta_t := ({\Delta_t^0},\Delta^1_t,\ldots,\Delta^m_t)$ is well-defined only for $t \in [0,T_1]$, being $T_1$ the first maturity of our forwards: thus, we assume that our agent trades in the market only until a final time $T$, with $T \leq T_1$.  As a consequence, we will give all the relevant definitions about portfolio, arbitrage and completeness relative to the final horizon $T$. 

Starting from the discussion above,
we assume that an agent builds a financial portfolio by investing at time $t \in [0,T]$, with $T \leq T_1$, {the quantity $\Delta_t^0$ in the riskless asset, and} the quantity $\Delta_t^j$ in the forward contract with maturity $T_j$, $j=1, \ldots, m$.  This results in a {\bf portfolio value} defined as $X_t^\Delta := {\Delta_t^0 +} \sum_{j=1}^n \Delta^j_t F(t,T_j)$. 
We call this portfolio {\bf self-financing} if the dynamics of its value is 
\begin{equation} \label{dX1}
dX_t^\Delta = \sum_{j=1}^n \Delta^j_t dF(t,T_j), \qquad t \in (0,T], 
\end{equation}
with initial condition $X_0^\Delta := x  \geq  0$. 
 %
In this definition, 
we implicitly
assume that  $(\Delta_t^j \mu(t,T_j))_t$ and $(\Delta^j_t  \sigma(t,T_j))_t$ are progressively measurable processes with sample paths belonging to $L^1([0,T])$ and $L^2
([0,T])$ for almost all $\omega$, respectively.  
We call a self-financing portfolio {\bf admissible} if there exists $C \in \bR$ such that $X_t^\Delta \geq C$ for almost all $(t,\omega) \in [0,T] \times \Omega$. Finally, an admissible portfolio is called an {\bf arbitrage} if $X_0^\Delta \equiv 0$, $\bP\{X_T^\Delta \ge 0\} = 1$ and $\bP\{X_T^\Delta > 0\} > 0$. 

 In order to rule out the possibility of arbitrage, in the next section we present the classical criterion of existence of an equivalent martingale measure. 

\subsection{ Absence of arbitrage}\label{aac}



Let us consider a function $\theta:[0,\bar T] \to \bR^n$
such that 
\begin{equation*}\label{theta}
    \int_0^T |\theta|^2_s ds < \infty.
\end{equation*}
Also,  consider another probability measure $\bQ \sim \bP$ on $(\Omega, \cF_{\bar T} )$ whose density process 
with respect to $\bP$ for any $t \in [0,\bar T]$ 
is given by 
\begin{equation} \label{ZT}
Z_t = \left. \frac{d \bQ}{d \bP} \right|_{\cF_t}
=\exp{\left( -\int_0^t \theta_s \cdot d W_s -\frac{1}{2} \int_0^t |\theta_s|^2  ds \right)}, 
\end{equation}
where 
$\int_0^t \theta_s \cdot d W_s = \sum_{i=1}^n \int_0^t \theta^i_s d W^i_s$.
 If this measure $\bQ$ is such that the price of all traded forward contracts in the market are $\bQ$-martingales, then $\bQ$ is an equivalent martingale measure (EMM) and no arbitrage opportunities exist  in the market \cite[Lemma 12.1.6]{Oksendal}.
In particular, for forward prices $F(\cdot,T_j)$, $j=1,\ldots,m$, this, together with the terminal condition $F(T_j,T_j) = S_{T_j}$, implies that each forward price satisfies
\begin{equation} \label{S-F}
F(t,T_j) = \bE_{\bQ} \left[ S_{T_j}| \cF_t\right],  \quad t \in [0,T_j],  \quad j=1,\ldots,m. 
\end{equation}
 Thus, we have a pricing relation between the spot and
forward price, that leads to an arbitrage-free pricing dynamics for the forward
price. 

 Now, we provide a necessary and sufficient condition for the absence of arbitrage in the market. 
\begin{proposition}  \label{proposition1}
Let $K(T_j,t)$ be defined by Equation \eqref{Kvec} and $\bQ \sim \bP$ be defined by Equation \eqref{ZT}  with $\int_0^{\bar T} |\theta|^2_s ds < \infty$. 
Then the process $W^\bQ = \{ W^\bQ_t, t \in [0,\bar T]\}$, defined as
\begin{equation} \label{WQ}
W_t^{\bQ} = W_t+\int_0^t \theta_s ds, \qquad t \in [0,\bar T]
\end{equation} 
is a Brownian motion on $\left(\Omega, \cF,(\cF_t)_{t  \in [0,\bar T]}, \bQ \right)$, and we have that
\begin{eqnarray} 
S_{t} 
& = &  \int_0^{t} K(t,s) \cdot dW_s^\bQ + \varphi_\bQ(t), \label{newphi2}
\end{eqnarray} 
where 
\begin{equation} \label{phiQ}
\varphi_\bQ(t) := \varphi(t)-\int_0^{t} K(t,s) \cdot \theta_s ds
\end{equation}
is the new deterministic seasonality function under $\bQ$. 
Moreover,  there is no arbitrage in the market, i.e. Equation \eqref{S-F} holds, if and only if the dynamics of $F(\cdot,T_j)$ under $\bQ$ is 
\begin{equation} \label{FQ}
d F(t,T_j) = K(T_j,t) \cdot dW_t^{\bQ}, \quad t \in [0,T_j], \qquad F(T_j,T_j) = S_{T_j}, 
\end{equation}
and the dynamics of each $F(\cdot,T_j)$, $j = 1,\ldots,m$, under $\bP$ is given by Equation \eqref{FP}, where 
\begin{equation}\label{musigma}
 \mu(t,T_j) = K(T_j,t) \cdot \theta_t \quad \mbox{ and } \quad \sigma(t,T_j) = K(T_j,t), \quad \text{for a.a} \quad t \leq T_j. 
\end{equation}
\end{proposition}

 The proof of this proposition is quite standard and follows in part that of \cite[Theorem 12.1.8]{Oksendal}. For the readers' convenience, we present a version adapted to our case in the Appendix. 

\begin{remark}
Let us note that $\left\{F(t,T_j), t \in [0,T_j] \right\}$ as solution of \eqref{dFP} is a Markov process. Thus,  starting from the spot price being a possibly non-Markovian and non-semimartingale process, as is well known e.g. for fractional Brownian motion, we obtain that the forward prices with any fixed maturity $T_j$, $j= 1,\dots,n$, are Gaussian and Markov martingales under the measure $\bQ$. 
\end{remark}

\begin{remark} 
Equation \eqref{musigma} states that $\mu$ and $\sigma$, which in principle are defined only for their second argument equal to the  observed maturities $T_j$,  $j = 1,\ldots,m$, are instead  defined structurally as functions of the vector of Volterra kernels $K$ (and $\theta$), which instead are defined for all  times, also different from $T_j$. For this reason, if Equation  \eqref{musigma} is true, we can extend the definition of $\mu$ and $\sigma$ for $T$ possibly different from $T_j$,  $j = 1,\ldots,n$, as 
%
%
$$ \mu(t,T) = K(T,t) \cdot \theta_t \qquad \mbox{ and } \qquad \sigma(t,T) = K(T,t). $$
This allows us to write the seasonality function $\varphi_\bQ(T)$ under $\bQ$, by comparing Equation \eqref{phiQ} with \eqref{musigma}, as 
\begin{equation} \label{newphifinal}
\varphi_\bQ(T) := \varphi(T) - \int_0^T \mu(t,T) dt. 
\end{equation}
\end{remark}

\subsection{Completeness}\label{sec2.2} 

From here on, we assume that the market is arbitrage free, {thus, at least one EMM $\bQ$ exists}: by virtue of Proposition \ref{proposition1}, this implies that there exists a {deterministic} function $\theta:[0,\bar T] \to \bR^n$ 
such that $\int_0^{\bar T} |\theta|^2_s ds < \infty$
and that satisfies 
\begin{equation}\label{sysmu}
   \bar \mu(t) = \bar K(t) \theta_t \quad \text{and} \quad \bar \sigma(t) = \bar K(t) \quad \text{for a.a.} \quad  {t \leq T_1},
   \end{equation}
 where
$$ \bar \mu(t) := \left( \begin{array}{c} \mu(t,T_1) \\ \vdots \\ \mu(t,T_m) \end{array} \right), \quad 
\quad \bar \sigma(t) := \left( \begin{array}{c} \sigma(t,T_1) \\ \vdots \\ \sigma(t,T_m) \end{array} \right),  $$
and
\begin{equation}\label{matrixK}
{\bar K(t)}=(K_i(T_j,t))^{i=1,\ldots,n}_{j=1,\ldots,m} = \left( \begin{array}{ccc}
K_1(T_1,t)  & \ldots    & K_n(T_1,t) \\
\vdots      & \ddots    & \vdots \\
K_1(T_m,t)  & \ldots    & K_n(T_m,t) \\
\end{array} \right) . 
\end{equation} 

 At this point, we want to investigate the completeness of our market. We follow \cite{Oksendal} and say that our market is {\bf complete} if, for every possible payoff $Y \in L^\infty(\Omega,{\cal F}_T,\bP)$, we can find a real number  $y$  and a portfolio $\Delta = (\Delta^1,\ldots,\Delta^n)$ such that 
\begin{equation} \label{completeness}
 Y  =   y  + \sum_{j=1}^n \int_0^T \Delta^j_t dF(t,T_j) \quad  {a.s.}
\end{equation} 
{and such that the process
\begin{equation*}
 y  + \sum_{j=1}^n \int_0^s \Delta^j_t\cdot K(T_j,t) dW^{\bQ}_t, \qquad 0 \leq s \leq T, 
\end{equation*}
is a $\bQ$-martingale.}

\begin{remark}
If the market is complete, then the representation in Equation \eqref{completeness} usually can be extended to payoffs  $Y$  which are also non-bounded, see e.g. \blue{\cite[Exercise 12.3]{Oksendal}} for an extension to payoffs $ Y  \in L^2(\Omega,{\cal F}_T,\bP)$.
\end{remark}

 \begin{theorem}\label{complete1}
The market $(F(\cdot,T_j))_{j=1,\ldots,m}$, 
is complete if and only if $\bar K(t):\bR^n \to \bR^m$ 
has a left inverse \blue{$\bar K^{-1}_{\mathrm{left}}(t)$}, i.e. there exists a matrix $\bar K^{-1}_{\mathrm{left}}(t):\bR^m \to \bR^n$ such that
\begin{equation}\label{leftinv}
\blue{\bar K^{-1}_{\mathrm{left}}(t)\bar K(t)}= I_m \quad \text{for a.a.} \quad 
{{t \leq T}},
\end{equation}
 where $I_m$ is the identity matrix of order $m$.
\end{theorem}
\begin{proof}
We refer to \cite[Theorem 12.2.5]{Oksendal}.
\end{proof}

 Let us note that, by virtue of the previous theorem, if the market is complete then the function $\theta_t$ satisfying \eqref{sysmu} is unique. Indeed, the only solution is given by $$\theta_t=\blue{\bar K_{\mathrm{left}}^{-1}(t)}\bar \mu(t).$$
Thus, we have a unique EMM $\bQ$.
\begin{remark}
    The existence of the left inverse for which equation \eqref{leftinv} holds is equivalent to the injectivity of $\bar K(t)$ for a.a. 
    {$t \leq T$}.
    Thus, \blue{$\mathrm{rank}(\bar K(t))=n$} for a.a. $t \leq T_1$, in particular, $m \geq n$.
\end{remark}
\begin{remark}
     We require, for convenience, the injectivity of $\bar K(t)$ 
     {``for a.a. $t \leq T$''.}
     However, we can note that completeness could still hold for 
     {$T > T_1$}
     as long as the remaining forward contracts span a diffusion matrix (that is $\bar K(t)$ with reduced number of rows) with rank equal to $n \leq m$.  Thus, in this case, even if we have assumed, from the beginning, that we are interested in trading only until 
     {$T\leq T_1$}
     in principle, we can trade even after some contracts have reached maturity, continuing to have a complete market.  
 \end{remark}
\begin{corollary}\label{complete2}
If $m=n$ the market is complete if and only if $\bar K(t)$ is invertible for a.a. 
{$t \leq T$}.
\end{corollary}

\begin{remark}
We underline that the forward prices $F(t,T_j)$, $j = 1,\ldots,m$, that we model here are only an approximation (or better, the building blocks) of the real contracts that are traded in electricity markets. Those last can be represented by the quantities $F(t,T_j,T_k)$, each one of these being the price at time $t$ of a forward contract\footnote{also called swap contracts in this framework by some authors} which delivers a fixed intensity of electricity over the period $[T_j,T_k]$, where $0\le t \le T_j < T_k \le \bar T$. More precisely, according to \cite{BPV} and to the no-arbitrage condition 
$$F(t,T_j,T_k)=\frac{1}{T_k-T_j}\int_{T_j}^{T_k} \bE_{\bQ}[S_T|\cF_t] dT = \frac{1}{T_k-T_j}\int_{T_j}^{T_k} F(t,T) dT, $$
valid when the settlement takes place at the end of the delivery period (this relation represents well also other kind of settlement rules, see Proposition 2.1 and Remark 1 in \cite{PSV} for details), its dynamics under $\bQ$ is
\begin{equation} \label{swap}
dF(t,T_j,T_k) = \bar K(t,T_j,T_k) \cdot dW^{\bQ}_t, 
\end{equation}
where the function $\bar K:\bR^3 \to \bR^n$ is defined as
\begin{equation} \label{barK}
\bar K(t,T_j,T_k)= \frac{1}{T_k-T_j} \int_{T_j}^{T_k} K(T,t) dT. 
\end{equation}
However, a common market practice when the delivery period $T_k - T_j$ is short (for example, for hourly or daily forward contracts) is to model not directly the traded quantities $F(t,T_j,T_k)$, but rather to model $F(t,\tilde T)$ with $\tilde T \in [T_j,T_k]$ as a proxy (for example, $\tilde T$ being the middle point of $[T_j,T_k]$). This market practice is justified also in our framework, as the next lemma shows.
\end{remark}
\begin{lemma} \label{2H}
If the dynamics of $F(t,\tilde T)$ and $F(t,T_j,T_k)$ are given respectively by Equations \eqref{FQ} and \eqref{swap}, with $\tilde T \in [T_j,T_k]$, and the kernel $K(T,t)$ is H\"older continuous in $T \in [0,\bar T]$ with exponent $\rho \in (0,1]$ uniformly in $t \in [0,\bar T]$, then 
$$ \mathrm{Var}[F(t,\tilde T) - F(t,T_j,T_k)] = O((T_k - T_j)^{2\rho}). $$
\end{lemma}
\begin{proof}
Since both $F(t,\tilde T)$ and $F(t,T_j,T_k)$ are defined by Wiener integrals, we have
$$ \mathrm{Var}[F(t,\tilde T) - F(t,T_j,T_k)] = \int_0^t |K(\tilde T,s) - \bar K(s,T_j,T_k)|^2 ds. $$
Assume now that $|K(T,s) - K(T',s)| \leq C |T - T'|^\rho$ for all $T,T' \in [0,\bar T]$ and for a suitable $C$. Then, from Equation \eqref{barK} we have that
\begin{eqnarray*}
\lefteqn{|K(\tilde T,s) - \bar K(s,T_j,T_k)| = \left| \frac{1}{T_k - T_j} \int_{T_j}^{T_k} (K(\tilde T,s) - K(T,s)) 
 {dT}
\right| \leq } \\
& \leq & \frac{1}{T_k - T_j} \int_{T_j}^{T_k} C|\tilde T - T|^\rho dT = \frac{C}{(1 + \rho)(T_k - T_j)} [(\tilde T - T_j)^{1 + \rho} + (T_k - \tilde T)^{1 + \rho}] \leq \\
& \leq & \frac{C}{(1 + \rho)(T_k - T_j)} [(T_k - T_j)^{1 + \rho} + (T_k - T_j)^{1 + \rho}] =  \frac{2 C}{1 + \rho} (T_k - T_j)^\rho.
\end{eqnarray*}
Thus,
\begin{eqnarray*}
\mathrm{Var}[F(t,\tilde T) - F(t,T_j,T_k)] & \leq & \int_0^t \left(\frac{2 C}{1 + \rho} (T_k - T_j)^\rho\right)^2 ds \leq \\
& \leq & \frac{4 C^2 \bar T}{(1 + \rho)^2} (T_k - T_j)^{2\rho},
\end{eqnarray*}
which yields the conclusion.
\end{proof}

\begin{remark}
While the case when the delivery period $T_k - T_j$ is short can be treated with the aid of Lemma \ref{2H}, this is not the case which concerns some forward contracts commonly met in electricity markets, where $T_k - T_j$ can be of order of one month, one trimester or one year. Instead, in the case of long delivery periods, when the approximation of Lemma 10 can possibly be not accurate enough, we can still use the results of Sections \ref{aac} and \ref{sec2.2}, but take care of substituting the kernels $K(T_j,t)$, relative to ``instantaneous'' forwards introduced in Equation \eqref{S-F}, with the kernels $\bar K(t,T_j,T_k)$ in Equation \eqref{barK}, relative to the forward contracts which are actually traded in the market. When  this is done, one should carry out the same kind of computation that we show in Sections \ref{GVP}, \ref{Sec:Optimal}, and \ref{Sec:Option} for the specific situation. However, the possible scenarios could be quite involved: in fact, the  ``classical'' assumption in academic literature is that these delivery period do not overlap (i.e. forward contracts are all of the same kind, e.g. monthly contracts); conversely, the typical situation in real market is that these contracts do have overlapping delivery periods (a discussion on why and how this happens can be found e.g. in \cite{LPV}), as the next example shows. In cases like this, one should also be careful to avoid what is defined in \cite{PSV} as  ``no overlapping arbitrage'' (NOA), i.e.~the possibility of building an arbitrage by considering forward contracts whose delivery period overlap (i.e., have a non-empty intersection). As this would deviate the focus from the main topic of this paper, we do not enter in further detail here.
\end{remark}
\begin{example}
Assume that we are at the beginning of May, 2024; then a typical situation is that the market quotes monthly forward prices for June, July and August 2024 (called respectively Jun/24, Jul/24 and Aug/24), for the 3rd and 4th trimester 2024 (Q3/24 and Q4/24) and for the calendar year 2025 (Cal-25), without further granularity until the end of May. This corresponds to setting $T_j$ as follows:
\begin{center}
\begin{tabular}{ll}
$T_j$   & actual date\\
\hline
$T_1$   & 01-06-2024  \\
$T_2$   & 01-07-2024  \\
$T_3$   & 01-08-2024  \\
$T_4$   & 01-09-2024  \\
$T_5$   & 01-10-2024  \\
$T_6$   & 01-01-2025  \\
$T_7$   & 01-01-2026  \\
\end{tabular}
\end{center}
and the quoted forward prices as 
\begin{center}
\begin{tabular}{rc}
market name & corresponding stochastic process\\
\hline
Jun/24  & $F(t,T_1,T_2)$\\
Jul/24  & $F(t,T_2,T_3)$\\
Aug/24  & $F(t,T_3,T_4)$\\
Q3/24   & $F(t,T_2,T_5)$\\
Q4/24   & $F(t,T_5,T_6)$\\
Cal-25  & $F(t,T_6,T_7)$\\
\end{tabular}
\end{center}
As we can see, the intervals $[T_j,T_k]$ overlap, as Jul/24 and Q3/24 have the same initial time and Aug/24 has the delivery period $[T_3,T_4]$ which is strictly contained in $[T_2,T_5]$, i.e. in the  delivery period of  Q3/24. 
\end{example}

 \section{ 
Gaussian Volterra Processes to model electricity spot prices} 
\label{GVP}

 In this section, we provide some examples of Gaussian Volterra processes which can be used to model electricity spot prices. We show analytically that, under conditions ensuring the absence of arbitrage, the completeness of the market is satisfied (see Theorem \ref{complete1}  and Corollary \ref{complete2}). While this is straightforward when the number of factors $n$ is 1, checking this for $n > 1$ involves proving that the matrix $\bar K(t)$ is non-singular for 
  a.a. $t \in [0,T]$. This is generally true, and quite straightforward to prove, when the kernels $K$ have different functional form from one another, resulting in factors which belong to different classes of Gaussian processes, 
  for example when we have a standard Ornstein-Uhlenbeck process and an  Ornstein-Uhlenbeck process driven by a  fractional Brownian motion  (fractional Ornstein-Uhlenbeck, fOU processes), see e.g. Section \ref{mixed}.
   Less trivial is the case when the Gaussian processes belong to the same class, for example when we have two fractional Brownian motion with different Hurst exponents. For this reason, in the following sections we check the most common cases that one could face, namely the case of Riemann-Liouville (RL) processes, and the case of a fBm, in both the cases with Hurst exponent $H > 1/2$ and $H < 1/2$. We also show how to formulate these results for Ornstein-Uhlenbeck (OU) processes driven by them, obtaining the original driving processes in the case when the mean-reversion speed is zero: for this reason, we first present how to treat OU processes driven by a generic Gaussian Volterra process, and then we will analyze the cases when the OU processes are driven by Gaussian-Volterra processes of the same kind as seen above, namely RL processes and  fBm with $H > 1/2$. We will analyze in detail the case when one has only $n = 2$ factors, as all the cases with $n > 2$ get more and more complicated 
 to handle. The only exception is when we have a generic number $n$ of RL fBm, in which case market completeness is equivalent to the determinant of a generalized Vandermonde matrix being nonzero, which we check to be true.   


\subsection{From Gaussian Volterra processes to Gaussian Volterra-driven Ornstein-Uhlenbeck processes}

Consider the Gaussian Volterra process
\begin{equation} \label{GV}
Z_t=\int_0^t K_Z(t,s) dW_s,
\end{equation}
where $W$ is a one dimensional Wiener process and $K_Z$ is a kernel such that $\int_0^t K^2_Z(t,s) ds < \infty$ for any $t>0$. Under this assumption $\bE Z_t^2 <\infty$ for any $t>0$. 
Consider the Langevin equation of the form
\begin{equation} \label{GVOU}
Y_t= \alpha\int_0^t Y_s ds + Z_t, \qquad t \geq 0, 
\end{equation}
 where we search for a solution $Y = \{Y_t, t\ge 0\}$ having Lebesgue-integrable sample paths, so that the integral $U_t := \int_0^t Y_s ds$ in the right-hand side is well defined.  Denoting $U_t=\int_0^t Y_s ds$, we can rewrite the above equation  
as 
\begin{equation}\label{U}
U_t' - \alpha U_t = Z_t,  
\end{equation}
 as the process $U$ has continuous sample paths which are a.e.   differentiable with respect to  the Lebesgue measure. Moreover, if $Z$ is continuous a.s., then equation \eqref{U} holds for all $t >0$.
  The condition of continuity of $Z$ is based on the standard Kolmogorov continuity theorem: 
we here follow \cite{MSS}, which specifically analyzes Gaussian Volterra processes.  Let $Z=\{Z_t, t \geq 0\}$ be a centered Gaussian process. If there exists 
 \blue{$C>0$} and $\delta >0$ such that 
\begin{equation}\label{cond1}
\bE |Z_t-Z_s|^2 \leq \blue{C}|t-s|^{\delta} \qquad \text{for all} \quad 0 \leq s \leq t \leq T, 
\end{equation}
then the process $Z$ has a modification that is continuous on $[0,T]$ 
which, moreover, satisfies H\"older continuity on $[0,T]$ of any order $0< \gamma < \frac{\delta}{2}$. 
 
\begin{remark}    
In the general case of a Gaussian Volterra process $Z$ defined as in Equation \eqref{GV}, the condition \eqref{cond1} is reduced to (see e.g.  \cite{MSS})
\begin{equation*}\label{cond2}
\int_0^s (K_Z(t,u)-K_Z(s,u))^2 du + \int_s^t K_Z^2(t,u) du \leq \blue{C}|t-s|^{\delta}.
\end{equation*}
\end{remark}
\begin{proposition}\label{prop17}
Let the condition \eqref{cond1} be fulfilled. Then Equation \eqref{U} has the unique continuous solution
\begin{equation} \label{Usol}
U_t=e^{\alpha t} \int_0^t e^{-\alpha s} Z_s ds, 
\end{equation}
and Equation \eqref{GVOU} has the unique continuous solution
\begin{equation}\label{Yt}
Y_t= U'_t= \alpha e^{\alpha t} \int_0^t e^{-\alpha s} Z_s ds + Z_t,
\end{equation}
which can be represented as a Gaussian-Volterra process as
\begin{align*}
Y_t & =  \int_0^t K_Y(t,u) dW_u,
\end{align*}
with 
 \begin{equation}\label{KY}
 K_Y(t,u)= \alpha \int_u^t e^{\alpha(t-s)}K_Z(s,u) ds + K_Z(t,u).    
 \end{equation} 
\end{proposition}
\begin{proof}
If the condition \eqref{cond1} holds, then Equation \eqref{U} is a linear ordinary differential equation in $U$ with non-homogeneous term $Z$, which is a continuous random function (here we consider the continuous modification of $Z$). Thus, the standard theory of ordinary differential equations gives the unique continuous solution as in Equation \eqref{Usol}, which is differentiable and has derivative given by Equation \eqref{Yt}. This is obviously solution of Equation \eqref{GVOU}, and it is very easy to see that this equation admits a unique solution. Take in fact two solutions $Y^1$ and $Y^2$: then we have that
$$ Y^1_t - Y^2_t = \alpha \int_0^t (Y^1_s - Y^2_s) ds $$
which, by the Gronwall lemma, has as consequence $Y^1_t - Y^2_t \equiv 0$ for all $t > 0$, thus, $Y^1 \equiv Y^2$, and the solution is unique. According to \cite[Chapter 4, Theorem 64]{Protter}, we can apply the stochastic Fubini theorem to the first term in the right-hand side of \eqref{Yt} and get that 
$$ Y_t = \alpha e^{\alpha t} \int_0^t e^{-\alpha s} \int_0^s K_Z(s,u) dW_u ds + \int_0^t K_Z(t,u) dW_u= 
 \int_0^t K_Y(t,u) dW_u, $$
where $K_Y$ is defined as in Equation \eqref{KY}: this concludes the proof.
\end{proof}

%
%

The advantage of \eqref{KY} is that we immediately get it avoiding any other operator or transformation, fractional or not. However, if we want to reduce \eqref{KY} to one term, then some additional assumptions are needed.
 
\begin{lemma} \label{lemmaKY}
If $K_Z$ is such that $K_Z(u,u) = 0$ and is differentiable in the first variable, with $\frac{\partial K_Z}{\partial s}$  Lebesgue integrable in $s$, then
\begin{equation}\label{diff}
K_Y(t,u) = \int_u^t e^{\alpha (t-s)} \frac{\partial K_Z(s,u)}{\partial s}  ds    
\end{equation}  
\end{lemma}
\begin{proof}
Under the previous assumptions, we can integrate in \eqref{KY} by parts and get that
\begin{align*}
K_Y(t,u)& = e^{\alpha t} \left( - e^{-\alpha s} K_Z(s,u) {\vert_u^t} + \int_u^t e^{-\alpha s} \frac{\partial K_Z(s,u)}{\partial s} \right) ds +K_Z(t,u)\\
& = - K_Z(t,u) + e^{\alpha (t-u)} K_Z(u,u) + \int_u^t e^{\alpha (t-s)} \frac{\partial K_Z(s,u)}{\partial s}  ds + K_Z(t,u)  
\end{align*}
and the result follows.
\end{proof}
\begin{remark} If to relate \eqref{KY} and \eqref{diff} to fractional Brownian motion and fractional Ornstein-Uhlenbeck process, then  \eqref{KY} is valid for any $H\in(0,1)$, while \eqref{diff} is valid only for  $H\in(1/2,1)$. More detail is given below. More general kernels are studied in \cite{MSS}.
\end{remark}

Now, let us consider some examples of Gaussian-Volterra processes. 

\subsection{Riemann-Liouville processes}

The Riemann-Liouville (RL) process 
is defined as 
\begin{equation} \label{RLfBm} 
   U^{H}_t =    c_{H,U}   \int_0^t (t-s)^{H-\frac{1}{2}} d W_s 
\end{equation}
where $W$ is a standard 1-dimensional Brownian motion, and normalizing factor $c_{H,U}$, by analogy with fractional integrals, is chosen as  $c_{H,U} := \frac{1}{\Gamma(H+\frac{1}{2})}$.\\ So, the RL process  
is a Gaussian-Volterra process having the  kernel $K_{ U^H}(t,s) :=  c_{H,U}(t - s)^{H-\frac{1}{2}}$. This kernel is square-integrable in $[0,t]$ for any $H > 0$, and it is well known that $U^H$ has continuous sample paths for all $H > 0$. However, the conditions that we gave to have an OU process driven by this process with the good properties seen in Section 5.1 are valid only when $H > 1/2$. In fact, we have that $ K_{U^H}(s,s)  = 0$ if and only if $H > 1/2$; in this case, 
$$ \frac{\partial K_{ U^H}(t,s)}{\partial t} = c_{H,U} \left(H - \frac12 \right) (t - s)^{H - \frac32} $$
which is integrable in $t$ for $H > 1/2$. For this reason, we will analyze OU processes driven by RL processes only in the case $H > 1/2$, while we have a much more general result for RL processes.  We underline that the increments of RL processes are not stationary \cite{JC1}.  

\subsubsection{Riemann-Liouville processes, $n$ factors}

As announced, in the case of RL processes, we can prove that $n$ different factors, each one with a different Hurst exponent, make the matrix $\bar K(t)$ nonsingular for all $t > 0$. 
Consider $n$ independent Riemann-Liouville processes represented as in Equation \eqref{RLfBm},  driven by $n$ independent standard 1-dimensional Brownian motions. We assume that the Hurst exponent $H_i$, $i = 1,\ldots,n$, are all different; thus, without loss of generality, we can assume that 
$0 < H_1 < H_2 < \ldots < H_n < 1$.  
%
%

%
 From Corollary \ref{complete2}, a 
sufficient and necessary condition for the completeness of the market in this case is 

\begin{equation}\label{RlnMatrix}
\begin{vmatrix}
  {c_{H_1,U}}  (T_1-t)^{H_1-\frac{1}{2}} & \ldots &   {c_{H_n,U}}  (T_1-t)^{H_n-\frac{1}{2}}   \\ 
\vdots & \vdots & \vdots \\ 
{c_{H_1,U}}  (T_n-t)^{H_1-\frac{1}{2}}  & \ldots &   {c_{H_n,U}}  (T_n-t)^{H_n-\frac{1}{2}} 
\end{vmatrix}
\neq 0
\end{equation}
 for any $0 \leq t \leq T$, where $T < T_1 < \ldots < T_n$.  
Let us note that, as well as in the subsequent subsections, in order to assess whether the matrix \eqref{matrixK} is invertible we can neglect the normalizing constants 
appearing in each line.  
 Without loss of generality, we assume that $0 < H_1 < H_2 < \ldots < H_n < 1$. Then, by letting $\alpha_i := H_i - \frac12$ and $x_j := T_j - t$, for $i,j = 1,\ldots,n$, the following result is instrumental in proving the completeness.  

\begin{theorem}\label{Kmatrix}
Let us consider the determinant of the form

\begin{equation}\label{xiMatrix}
\Delta_n(x_1, \ldots, x_n)= 
\begin{vmatrix}
x_1^{\alpha_1} & \ldots & x_n^{\alpha_1}  \\ 
x_1^{\alpha_2} & \ldots & x_n^{\alpha_2} \\
\vdots & \vdots & \vdots \\ 
x_1^{\alpha_n} & \ldots & x_n^{\alpha_n}
\end{vmatrix}.
\end{equation}
 Then, for any $\alpha_i \in \mathbb R$ such that $\alpha_1 < \ldots < \alpha_n$ and $0 < x_1 \ldots < x_n,$ we have $\Delta_n >0$.
\end{theorem}

\blue{Let us note that the matrix in \eqref{xiMatrix} is the transpose of the matrix in \eqref{RlnMatrix}: clearly we can use the above result for our case, because the determinant is the same for a matrix and its transpose.
 Moreover,} let us note that the matrix in \eqref{xiMatrix} has the form of the so-called unsigned exponential Vandermonde matrix. The result on the positiveness of the determinant for the signed exponential Vandermonde matrix can be found in \cite{Robbin}. However, the proof for the case of the unsigned exponential matrix is not explicitly provided there, thus we report our proof in the Appendix \ref{appendix} for the readers' convenience. 

In our case, as $x_j = T_j - t$, we have $0 < T_1 - T \leq x_1 < \ldots < x_n \leq T_n$ for all $t \in [0,T]$,  
so we can see that the conditions  of Theorem \ref{Kmatrix} are satisfied. Thus, also the determinant of the matrix of the processes' kernels $\bar K(t)$ is positive for all $t \in [0,T]$. Therefore, in the case the electricity spot price is the sum of $n$ RL processes, we can assert that our market is complete, when the number $m$ of forward contracts is equal to $n$.
 Let us note that the completeness is still valid in the case $m > n$, by Theorem \ref{complete1}: indeed, thanks to Theorem \ref{Kmatrix}, we have  a minor of order $n$ different from zero, so the rank of the matrix is $n$

\subsubsection{ Ornstein-Uhlenbeck processes driven by RL processes, $H > 1/2$} 
 By Lemma \ref{lemmaKY}, an OU process with mean-reversion speed $\alpha$ driven by a RL fBm with Hurst exponent $H > \frac12$ is a Gaussian-Volterra process, with kernel given by
$$ K_Y(t,s) =  {c_{H,U}} \int_s^t e^{\alpha(t - u)} \left(H - \frac12\right) (u - s)^{H-\frac32} du. $$
\begin{proposition}
Consider two OU processes with the same mean-reversion speed $\alpha$, driven by two independent RL fBm with $\frac{1}{2} < H_1 < H_2$, and let $0 <T_1 < T_2$. In this case 
\blue{$$
\begin{vmatrix}
c_{H_1,RL}\int_s^{T_1} e^{\alpha(T_1-u)} (u - s)^{H_1-\frac32} du &c_{H_2,RL}\int_s^{T_1} e^{\alpha(T_1-u)} (u - s)^{H_2-\frac32} du  \; \\
c_{H_1,RL}\int_s^{T_2} e^{\alpha(T_2-u)} (u - s)^{H_1-\frac32} du  &  c_{H_2,RL}\int_s^{T_2} e^{\alpha(T_2-u)} (u - s)^{H_2-\frac32} du 
\end{vmatrix} \neq 0.$$}
\end{proposition}
\begin{proof}
We can prove, equivalently, that
$$   \Delta_{RL}(T_1,T_2) \neq 0,$$
where
\blue{$$\Delta_{RL}(T_1,T_2)=
\begin{vmatrix}
 \int_s^{T_1} e^{-\alpha u} (u - s)^{H_1-\frac32} du &  \int_s^{T_1} e^{-\alpha u} (u - s)^{H_2-\frac32} du \; \\
  \int_s^{T_2} e^{- \alpha u} (u - s)^{H_1-\frac32} du & \int_s^{T_2} e^{-\alpha u} (u - s)^{H_2-\frac32} du 
\end{vmatrix}.$$}
Obviously, $\Delta_{RL}(T_1,T_1)=0$, and 
\blue{$$\frac{\partial \Delta_{RL}(T_1,T_2)}{\partial T_2}=
e^{-\alpha T_2}\begin{vmatrix}
 \int_s^{T_1} e^{-\alpha u} (u - s)^{H_1-\frac32} du & \int_s^{T_1} e^{-\alpha u} (u - s)^{H_2-\frac32} du  \; \\
  (T_2 - s)^{H_1-\frac32}  & (T_2 - s)^{H_2-\frac32}   
\end{vmatrix} $$}
%
Consider now the difference
$$ \frac{(T_2 - s)^{H_2-\frac32}}{(T_2 - s)^{H_1-\frac32}} - \frac{\int_s^{T_1} e^{-\alpha u} (u - s)^{H_2-\frac32} du}{\int_s^{T_1} e^{-\alpha u} (u - s)^{H_1-\frac32} du} = 
\frac{(T_2 - s)^{H_2-\frac32}}{(T_2 - s)^{H_1-\frac32}} - \frac{e^{-\alpha \xi} (\xi - s)^{H_2-\frac32}}{e^{-\alpha \xi} (\xi - s)^{H_1-\frac32}} $$
for a suitable $\xi \in (s,T_1)$ to be chosen due to the Cauchy theorem about the ratio of increments of differentiable functions. Then we have 
$$ \frac{(\xi - s)^{H_2-\frac32}}{(\xi - s)^{H_1-\frac32}} = 
(\xi-s)^{H_2-H_1} < (T_2-s)^{H_2-H_1} = \frac{(T_2 - s)^{H_2-\frac32}}{(T_2 - s)^{H_1-\frac32}}, $$
which implies that $\frac{\partial \Delta_{RL}(T_1,T_2)}{\partial T_2} > 0$: this, joint to $\Delta_{RL}(T_1,T_1)=0$, implies that $\Delta_{RL}(T_1,T_2) > 0$ for all $T_2 > T_1$. 
\end{proof}

\subsection{ Ornstein-Uhlenbeck processes driven by fBMs, $H > 1/2$} 
Consider now the case of fBM $B^H$ with Hurst index $H \in (\frac{1}{2},1)$. 
Then $B^H$ admits the compact interval representation of the form
$$B^H_t=c_H \int_0^t s^{\frac{1}{2}-H} \int_s^t \varphi_H(u,s) du \; dW_s,$$
where $W$ is a Wiener process,  $$c_{H} := \left(H-\frac{1}{2} \right)\left( \frac{2 H \; \Gamma(\frac{3}{2}-H)}{\Gamma(H + \frac{1}{2})\Gamma(2-2H)}\right)^{\frac{1}{2}}$$  and 
$$ \varphi_H(u,s) := u^{H-\frac{1}{2}}(u-s)^{H-\frac{3}{2}}. $$
In other words, in this case the kernel is 
$$ K_{ {B^H}}(t,s) = c_H s^{\frac{1}{2}-H} \int_s^t \varphi_H(u,s) du $$
For $H > 1/2$, this kernel is square-integrable in $[0,t]$ and is such that $B^H$ has continuous (even H\"{o}lder up to order $H$) sample paths. Moreover, the conditions that we gave to have an OU process driven by this process with the good properties seen in Section 5.1 are valid. In fact, following Lemma \ref{lemmaKY}, we have that 
$ {K_{{B^H}}(s,s)} = 0$ and
$$ \frac{\partial K_{ {B^H}}(t,s)}{\partial t} = c_H s^{\frac{1}{2}-H} t^{H-\frac{1}{2}}(t-s)^{H-\frac{3}{2}} = c_H s^{\frac{1}{2}-H} \varphi_H(t,s) $$
which is integrable in $t$ for $H > 1/2$. Thus, we can also analyze OU processes driven by fBm. After Lemma \ref{lemmaKY}, an OU process with mean-reversion speed $\alpha$ lead by a fBm is a Gaussian-Volterra process with kernel
$$ K_Y(t,s) := c_H s^{\frac{1}{2}-H} \int_s^t e^{\alpha(t-u)} \varphi_H(u,s) du. $$
\begin{proposition}
Let us consider two {fOU} process with the same mean-reversion speed $\alpha$ and $\frac{1}{2} < H_1 < H_2$, and consider also $0 <T_1 < T_2$. Then, we have
\blue{$$
\begin{vmatrix}
 c_{H_1}s^{\frac{1}{2}-H_1} \int_s^{T_1} e^{\alpha(T_1-u)} \varphi_{H_1}(u,s) du &   c_{H_2}s^{\frac{1}{2}-H_2} \int_s^{T_1} e^{\alpha(T_1-u)} \varphi_{H_2}(u,s) du \; \\
 c_{H_1}s^{\frac{1}{2}-H_1} \int_s^{T_2} e^{\alpha(T_2-u)} \varphi_{H_1}(u,s) du &  c_{H_2}s^{\frac{1}{2}-H_2} \int_s^{T_2} e^{\alpha(T_2-u)} \varphi_{H_2}(u,s) du 
\end{vmatrix} \neq 0.$$}
\end{proposition}
\begin{proof}
We can prove, equivalently, that
$$  \Delta(T_1,T_2) \neq 0,$$
where
\blue{
$$\Delta(T_1,T_2)=
\begin{vmatrix}
 \int_u^{T_1} e^{-\alpha s} \varphi_{H_1}(s,u) ds &  \int_u^{T_1} e^{-\alpha s} \varphi_{H_2}(s,u) ds \; \\
 \int_u^{T_2} e^{- \alpha s} \varphi_{H_1}(s,u) ds  & \int_u^{T_2} e^{-\alpha s} \varphi_{H_2}(s,u) ds 
\end{vmatrix}.$$}
Obviously, $\Delta(T_1,T_1)=0$, and 
\blue{
$$\frac{\partial \Delta(T_1,T_2)}{\partial T_2}=
e^{-\alpha T_2}\begin{vmatrix}
 \int_u^{T_1} e^{-\alpha s} \varphi_{H_1}(s,u) ds & \int_u^{T_1} e^{-\alpha s} \varphi_{H_2}(s,u) ds   \; \\
   \varphi_{H_1}(T_2,u) & \varphi_{H_2}(T_2,u)   
\end{vmatrix} $$}
%
Consider now the difference
$$ \frac{\varphi_{H_2}(T_2,u)}{\varphi_{H_1}(T_2,u)} - \frac{\int_u^{T_1} e^{-\alpha s} \varphi_{H_2}(s,u) ds}{\int_u^{T_1} e^{-\alpha s} \varphi_{H_1}(s,u) ds} = \frac{\varphi_{H_2}(T_2,u)}{\varphi_{H_1}(T_2,u)} - \frac{e^{-\alpha \xi} \varphi_{H_2}(\xi,u)}{e^{-\alpha \xi} \varphi_{H_1}(\xi,u)} $$
for a suitable $\xi \in (s,T_1)$ to be chosen due to the Cauchy theorem about the ratio of increments of differentiable functions. By definition of $\varphi_H$, we have 
$$ \frac{\varphi_{H_2}(\xi,s)}{\varphi_{H_1}(\xi,s)} = 
\xi^{H_2-H_1} (\xi-s)^{H_2-H_1} < T_2^{H_2-H_1}(T_2-s)^{H_2-H_1}=\frac{\varphi_{H_2}(T_2,s)}{\varphi_{H_1}(T_2,s)}, $$
which implies that $\frac{\partial \Delta(T_1,T_2)}{\partial T_2} > 0$: this, joint to $\Delta(T_1,T_1)=0$, implies that $\Delta(T_1,T_2) > 0$ for all $T_2 > T_1$. 
\end{proof}

 This result ensures that for the case in which the electricity price is driven by two fOU process with $\frac{1}{2}< H_1 < H_2$, the same mean-reversion speed $\alpha$ and $0 <T_1 < T_2$, the market is complete. Since for $\alpha = 0$ the two OU processes degenerate in two fractional Brownian motions, the result holds true also for this simpler case.   


\subsection{ Fractional Brownian motions, $H < 1/2$}






Now, let $Z=B^H$, a fBm with Hurst index $H \in (0,\frac{1}{2})$. In this case, we have the representation

 $$B_{t}^H=\int_0^{t} K^H_{B}(t,s) dW_s, \quad {t}\geq0$$
 for ease of notation in the following, we denote $K^H_{B}:=K_H$,
which explicitly can be written as 
\begin{align*}
K_H(t,s)&=\bar{c}_H \left[ \left(\frac{t}{s}\right)^{H-\frac{1}{2}}(t-s)^{H-\frac{1}{2}} +\left(\frac{1}{2}-H \right) s^{\frac{1}{2}-H}\int_s^t u^{H-\frac{3}{2}}(u-s)^{H-\frac{1}{2}} du\right]\\
&= \bar{c}_H s^{H-\frac{1}{2}} \left[ \left(\frac{t}{s}\right)^{H-\frac{1}{2}} \left( \frac{t}{s} -1\right)^{H-\frac{1}{2}} + \left(\frac{1}{2}-H \right)\int_1^{t/s} v^{H-\frac{3}{2}}(v-1)^{H-\frac{1}{2}} dv \right]\\
&= \bar{c}_H s^{H-\frac{1}{2}} \tilde K_H\left( \frac{t}{s},1\right),
\end{align*}
 where
 $$\bar{c}_H= \left( \frac{2 H \; \Gamma(\frac{3}{2}-H)}{\Gamma(H + \frac{1}{2})\Gamma(2-2H)}\right)^{\frac{1}{2}},$$ and
 \begin{equation}\label{tildeK}
 \tilde K_H(r,1)=  r^{H-\frac{1}{2}}(r-1)^{H-\frac{1}{2}}+\left( \frac{1}{2}-H\right)\int_1^r \varphi_H(v) dv,
 \end{equation}
with $\varphi_H(v)=v^{H-\frac{3}{2}}(v-1)^{H-\frac{1}{2}}$.  In the following, for brevity, we denote $  K_H(r) := \tilde K_H(r,1)$. Notice that
$$   K_H'(r) = \left( H - \frac12 \right) r^{H - \frac12} (r - 1)^{H - \frac32}. $$

Now, let $0<T < T_1 < T_2$ and $0 <H_1 < H_2 < \frac{1}{2}$. We want to investigate the behaviour of the matrix \eqref{matrixK}, that in this case writes:

$$\Delta(s)=
\begin{vmatrix}
 \bar{c}_{H_1}s^{H_1-\frac{1}{2}}   K_{H_1}(\frac{T_1}{s}) & \bar{c}_{H_2}s^{H_2-\frac{1}{2}}   K_{H_2}(\frac{T_1}{s})\\
  \bar{c}_{H_1}s^{H_1-\frac{1}{2}}  K_{H_1}(\frac{T_2}{s}) & \bar{c}_{H_2}s^{H_2-\frac{1}{2}}   K_{H_2}(\frac{T_2}{s})
\end{vmatrix},
$$
specifically, we wonder if there are any points $0< s \leq T$ where $\Delta(s)$ is zero. Thus, in the following, we proceed to answer this question. Equivalently, it is sufficient to analyze the behavior of the determinant
$$
\tilde \Delta(R_1,R_2)=
\begin{vmatrix}
   K_{H_1}(R_1) &    K_{H_2}(R_1)\\
    K_{H_1}(R_2) &   K_{H_2}(R_2)
\end{vmatrix},
$$
for $1 \leq R_1 \leq R_2$, where $R_1=\frac{T_1}{s}$ and $R_2=\frac{T_2}{s}$. 
Obviously, $\tilde \Delta (R_1, R_1)=0$, for any $R_1 >1$. Moreover, from \eqref{tildeK} it is evident that for any $R_2>1$, $\tilde \Delta(R_1,R_2) \rightarrow + \infty$ as $R_1 \searrow 1$. Also, it is evident that $\tilde \Delta(R_1,R_2)$ has no limit in $(1,1)$. 

To study the specifics of this manifold, 
let us calculate
\begin{align*}
\tilde \Delta(R_1,R_2)&= K_{H_1}(R_1)K_{H_2}(R_2)-K_{H_1}(R_2)K_{H_2}(R_1)\\
&=K_{H_2}(R_2)K_{H_2}(R_1)\left[ \frac{K_{H_1}(R_1)}{K_{H_2}(R_1)}-\frac{K_{H_1}(R_2)}{K_{H_2}(R_2)}\right].
\end{align*}
Let $f(R)=\frac{K_{H_1}(R)}{K_{H_2}(R)}$. So, we are interested in the sign of the difference $f(R_1)-f(R_2)$ for $1 < R_1 < R_2$.  
We have
\begin{equation*}
  \begin{split}
 f'(R)&=\left[ \left(H_1-\frac{1}{2}\right) R^{H_1-\frac{1}{2}}(R-1)^{H_1-\frac{3}{2}} K_{H_2}(R) \right.\\
 &\left.-\left(H_2-\frac{1}{2}\right)R^{H_2-\frac{1}{2}}
 (R-1)^{H_2-\frac{3}{2}}K_{H_1}(R)\right] \cdot (K_{H_2}(R))^{-2} \\
 &=\frac{R^{H_1-\frac{1}{2}}(R-1)^{H_1-\frac{3}{2}}}{(K_{H_2}(R))^{2}}\left[ \left( H_1-\frac{1}{2}\right)K_{H_2}(R) \right.\\
 & \left. -\left(H_2-\frac{1}{2} \right)R^{H_2-H_1}(R-1)^{H_2-H_1}K_{H_1}(R)\right].
  \end{split}  
\end{equation*}
Thus, we have to study the sign of the term in the square brackets that we denote by $I(R)$. By some computations, we obtain

\begin{equation}
\begin{split}
I(R)&=\left(H_1-H_2 \right)R^{H_2-\frac{1}{2}}(R-1)^{H_2-\frac{1}{2}}\\
&-\left(\frac{1}{2}-H_1\right)\left(\frac{1}{2}-H_2 \right)\int_1^R u^{H_2-\frac{3}{2}}(u-1)^{H_2-\frac{1}{2}} du\\
&+ \left(\frac{1}{2}-H_1\right)\left(\frac{1}{2}-H_2 \right) R^{H_2-H_1}(R-1)^{H_2-H_1} \int_1^R u^{H_1-\frac{3}{2}}(u-1)^{H_1-\frac{1}{2}} du.
\end{split}
\end{equation}
We can see that if $R \searrow 1$, then $I(R) \rightarrow - \infty$,  if $R \rightarrow + \infty$, then $I(R) \rightarrow + \infty$. Moreover, it is easy to see that $I(R)$ strictly increases. 
This means that for some $R^*$, $f'(R)>0$, when $1<R<R^*$, $f'(R^*)=0$, and $f'(R)>0$ for $R> R^*$. That is, $R^*$ is a minimum point of $f$.
Moreover, when $R \searrow 1$ , $f(R) \rightarrow +\infty$ and when $R \rightarrow + \infty$, $f(R) \rightarrow f_{\infty}$, where $$f_\infty=\frac{1/2 -H_1}{1/2-H_2}\frac{\int_1^{\infty} u^{H_1-\frac{3}{2}}(u-1)^{H_1-\frac{1}{2}}du}{\int_1^\infty u^{H_2-\frac{3}{2}}(u-1)^{H_2-\frac{1}{2}} du},$$
and by the sign of the derivative of $f$, we have that there exists $R_* < R^*$, for which $f(R^*)<f(R_*)=f_{\infty}$. 

Now, let us recall that we are interested in the sign of $f(R_1)-f(R_2)$ in order to determine the sign of $\tilde \Delta$. Fix $R_1$ and consider $f(R_1)-f(R_2)$ as a function of $R_2$. 
From the analysis above, we have that if $R_1 < R^*$, $f(R_2)$ decreases when $R_2$ increases from $R_1$ to $R^*$,  
then for $R_2>R^*$, $f(R_2)$ increases up to $f_\infty$.
Thus, 
\begin{itemize}
    \item if $R_1<R_*$,  
    $\tilde \Delta(R_1,R_2) >0$ for all $R_2 >R_1$, 
    \item if $R_* <R_1<R^*$, then $f(R_1)<f_\infty$ and there exists unique $R_2^0>R^*$ such that $\tilde \Delta(R_1,R_2)  >0$ for $R_2<R_2^0$, $\tilde \Delta(R_1,R_2)  < 0$ for $R_2 >R_2^0$ and $\tilde \Delta(R_1,R_2^0)=0$,
    \item if $R_1 \geq R^*$, then $\tilde \Delta(R_1,R_2)  <0$ for all $R_2>R_1$.
\end{itemize}

Now, let $  R_* <\bar R_1 <R^*$ and $\bar R_2 > \bar R_1$ such that $\tilde \Delta(\bar R_1, \bar R_2) <0$. It means that $\bar R_2 > R^*$ and for any $R_1$ such that $\bar R_2>R_1>\bar R_1$,  $f(R_1)<f(\bar R_2)$ holds. 
If $\tilde \Delta(\bar R_1, \bar R_2) <0$, it means that $  R_* <  \bar R_1  $, and then   $\tilde \Delta( R_1, \bar R_2) <0$ for all $\bar R_1< R_1 < \bar R_2$.

 Denote $D^+$ and $D^-$ the sets in the "triangle" bounded by the lines $R_1=1, R_1=R_2$, where $\tilde{\Delta}(R_1,R_2)  >0$ and $\tilde{\Delta}(R_1,R_2)  <0$, respectively. Then these sets have the form as in the Figure \ref{fig:comparethr1}.

\begin{figure}[h]
 \centering
    \begin{minipage}{0.4\textwidth}
   \hspace{-2.5cm}
    \begin{tikzpicture}[scale=0.45]
        \begin{axis}[
            xmin = 0, xmax = 5,
            ymin = 0, ymax = 5,
            axis lines=middle,
            y=2cm,
            x=2cm,
            xtick={1,1.5, 2.26},
            xticklabels={$1$,$R_*$, $R^*$},
            ytick={1},
            yticklabels={$1$}]
            \addplot[
                domain = 1:5, , name path=A
                ]{x};
            \addplot[
                samples=100,
                domain = 1.5:2.26,
                name path=B]{0.2/(x-1.5)+2};
            \draw[dashed] (0,1) -- (1,1);
            \draw[dashed] (1,0) -- (1,1);
            \draw[] (1,1) -- (1,5);
            \draw[dashed] (1.5,0) -- (1.5,5);
            \draw[dashed] (2.27,0) -- (2.26,2.26);
            \addplot[draw=none, name path=C] {5};
            \addplot[gray] fill between[of=A and B,soft clip={domain=1.5:2.26}];
            \addplot[gray] fill between[of=A and C, soft clip={domain=1:1.5}];
            \addplot[lightgray] fill between[of=B and C, soft clip={domain=1.60:5}];
            \node[] at (axis cs: 1.4,2.3) {\resizebox{0.75cm}{!}{$D^+$}};
            \node[] at (axis cs: 2.25,3) {\resizebox{0.75cm}{!}{$D^-$}};
        \end{axis}
    \end{tikzpicture}
    \caption{}\label{fig:comparethr1}
    \end{minipage}
   \begin{minipage}{0.4\textwidth}
   \hspace{-2.0cm}
    \begin{tikzpicture}[scale=0.45]
        \begin{axis}[
            xmin = 0, xmax = 5,
            ymin = 0, ymax = 5,
            axis lines=middle,
            y=2cm,
            x=2cm,
            xtick={1,1.5, 2.26},
            xticklabels={$1$,$R_*$, $R^*$},
            ytick={1, 1.5},
            yticklabels={$1$, $\frac{T_2}{T_1}$}]
            \addplot[
                domain = 0:5, , name path=D
                ]{1.5*x};
            \addplot[
                domain = 1:5, , name path=A
                ]{x};
            \addplot[
                samples=100,
                domain = 1.5:2.26,
                name path=B]{0.2/(x-1.5)+2};
            \draw[dashed] (0,1) -- (1,1);
            \draw[dashed] (1,0) -- (1,1);
            \draw[] (1,1) -- (1,5);
            \draw[dashed] (1.5,0) -- (1.5,5);
            \draw[dashed] (0,1.5) -- (1,1.5);
            \draw[dashed] (2.27,0) -- (2.26,2.26);
            \addplot[draw=none, name path=C] {5};
            \draw [<-, thick](2.62,2.2) -- (1.79,2.685);
            \filldraw[black] (1.79,2.685) circle (2pt);
            \node[] at (axis cs: 3.7,1.7) {$\tilde \Delta\left(\frac{T_1}{s}, \frac{T_2}{s}\right) = 0$};
            \addplot[white] fill between[of=A and C, soft clip={domain=1:1.5}];
        \end{axis}
     \end{tikzpicture}
    \caption{}
    \label{fig:comparethr2}
\end{minipage}
\end{figure}

Coming back to our original problem, we want now to understand how many $0<s\le T < T_1<T_2$ exist for which $\Delta(s)=0$, that is equivalent to count how many point $(R_1,R_2)$ with $R_1=\frac{T_1}{s}$ and  $R_2=R_1 \cdot \frac{T_2}{T_1}$ are such that $\tilde \Delta(R_1,R_2)=0$.
Note that if $s\rightarrow 0$, then, starting from some $s_0$ , $\tilde\Delta(\frac{T_1}{s},\frac{T_2}{s})<0$  for all $s<s_0$. Instead, if $s \rightarrow T$, we can have two cases. It can be $\Delta(T)\leq 0$ and then for all $0 <s < T$, $\Delta(s)<0$. Or it can be $\Delta(T)>0$, and in this case, by the analysis above, we can assert that there is a unique $0<s\le T $ such that $\Delta(s)=0$. It corresponds to the unique point $(\frac{T_1}{s}, \frac{T_2}{s})$ laying on the boundary of $D^-$ (see Figure \ref{fig:comparethr2}). So,
we have the following result.

\begin{proposition}
If $\Delta(T)<0$, then for all $0 <s \leq T$, $\Delta(s)<0$, on the other hand, if $\Delta(T) \geq 0$ there exists a unique $0<s\le T $ such that $\Delta(s)=0$.
\end{proposition}
 Thus, we have that the condition ensuring completeness of the market is satisfied (see Corollary \ref{complete2}), indeed the determinant of $\bar K(t)$ vanishes at most in one point. 

\subsection{Mixed case (OU and fOU with $H > 1/2$)} \label{mixed}
 
As we already said at the beginning of this section, the cases when the Gaussian-Volterra processes driving the spot price $S$ belong to different classes are quite straightforward to treat. As an example, we here report a result relative to the case when we have two factors, namely a standard OU process and a fOU process with $H > 1/2$. 
\begin{proposition}
Let us consider a standard OU process, with kernel $K_1(t,u) := e^{\alpha_1(t-u)}$, and a fOU process with $H > 1/2$, with kernel 
$$ K_2(t,u) :=  c_H u^{\frac{1}{2}-H} \int_u^t e^{\alpha_2 (t-s)} \varphi_H(s,u) ds,$$
 where, as before
$$\varphi_H(s,u)=s^{H-\frac{1}{2}}(s-u)^{H-\frac{3}{2}}.$$
 
If $\alpha_1 \leq \alpha_2$, then for $T_1 < T_2$ we have
\blue{$$ \begin{vmatrix}
e^{\alpha_1 (T_1-u)} &  u^{\frac{1}{2}-H} \int_u^{T_1} e^{\alpha_2 (T_1-s)} \varphi_{H}(s,u) ds \\
  e^{\alpha_1 (T_2-u)}  & u^{\frac{1}{2}-H} \int_u^{T_2} e^{\alpha_2 (T_2-s)} \varphi_{H}(s,u) ds 
\end{vmatrix} \neq 0.$$}
\end{proposition}
\begin{proof}
We can prove, equivalently, that
$$ \Delta(T_1,T_2) \neq 0,$$
where
\blue{$$ \Delta(T_1,T_2)=
\begin{vmatrix}
 e^{(\alpha_1 - \alpha_2) T_1} & \int_u^{T_1} e^{- \alpha_2 s} \varphi_{H}(s,u) ds \\
  e^{(\alpha_1 - \alpha_2) T_2} &  \int_u^{T_2} e^{- \alpha_2 s} \varphi_{H}(s,u) ds 
\end{vmatrix}.$$}
Now, $\Delta(T_1,T_2) \neq 0$ is equivalent to
$$ \frac{e^{(\alpha_1 - \alpha_2) T_1}}{e^{(\alpha_1 - \alpha_2) T_2}} \neq \frac{\int_u^{T_1} e^{-\alpha_2 s} \varphi_{H}(s,u) ds}{\int_u^{T_2} e^{-\alpha_2 s} \varphi_{H}(s,u) ds},  $$
which is true if $\alpha_1 \leq \alpha_2$, as in this case 
$$ \frac{e^{(\alpha_1 - \alpha_2) T_1}}{e^{(\alpha_1 - \alpha_2) T_2}} \geq 1 >  \frac{\int_u^{T_1} e^{-\alpha_2 s} \varphi_{H}(s,u) ds}{\int_u^{T_2} e^{-\alpha_2 s} \varphi_{H}(s,u) ds}.  $$
\end{proof}
\blue{After these results related to the absence of arbitrage and completeness of our model, in the following sections we will face some relevant problems in {electricity} 
markets, as the portfolio optimization problem and the option pricing.} 

\blue{\section{Optimal investment}\label{Sec:Optimal}}

Now we want to solve the problem of an agent who can invest in the electricity market and wants to maximize the expected utility of her/his wealth at the terminal time $T \in [0,T_1]$. The dynamics for the portfolio wealth has already been given in Equation \eqref{dX1}, which we can rewrite as
\begin{equation} \label{dX}
dX_t^\Delta = \sum_{j=1}^n \Delta^j_t dF(t,T_j)=  \sum_{j=1}^n \Delta^j_t K(T_j,t) \cdot dW_t^\bQ = \Delta_t \bar K(t)  dW_t^\bQ. 
\end{equation}
We assume that no arbitrage exists in our market, i.e. there exists an equivalent martingale measure $\bQ$. 
More precisely, the agent wants to  solve the problem
\begin{equation}\label{primal}
\sup_{\Delta} \bE_{\bP}[u(X_T^{\Delta})], 
\end{equation}
where $u$ is a known 
utility function, i.e. a real function which is non-decreasing and concave, and  $\Delta$ is chosen among the admissible strategies, as defined in Section 2, with the additional requirements that  $u(X_T^{\Delta})$ is well-defined and such that Equation \eqref{dX}, with initial condition $X_0^\Delta := x > 0$,  has a unique strong solution $X^\Delta$. 
  We thus restrict our admissible strategies to the class satisfying also the additional assumptions above, and denote this new class of admissible strategy by $\mathcal{A}$. 

 Since each $F(\cdot,T_j)$, $j = 1,\ldots,m$, is a martingale under any equivalent martingale measure $\bQ$, Equation \eqref{dX} makes $X^\Delta$ a local martingale, which is also bounded from below if $\Delta$ is an admissible strategy: this implies that $X^\Delta$ is a supermartingale under any equivalent martingale measure $\bQ$. Since $X_0^\Delta = x$, this implies that 
\begin{equation} \label{admissible}
\bE_{\bQ}[X_T^\Delta] \leq x. 
\end{equation}

\begin{remark}
In the case when the utility function is well-defined only on non-negative or strictly positive wealth, as in the cases $U(x) = x^\gamma$, $\gamma \in (0,1)$ or $U(x) = \log x$, respectively, requiring that $u(X_T^{\Delta})$ is well-defined is equivalent to require that $X_T^{\Delta} \geq 0$ or $X_T^{\Delta} > 0$   a.s., respectively.  The first requirement simply corresponds to take the constant in the definition of admissibility being equal to zero, while the second one poses a slight restriction to this condition. In both cases,  by the properties of conditional expectation, the supermartingale property for $X^\Delta$ implies that also $X_t^{\Delta} \geq 0$ or $X_t^{\Delta} > 0$  a.s.  for all $t \in [0,T]$, respectively. 
\end{remark}

 



From now on, we assume that the market is complete, so that the equivalent martingale measure $\bQ$ is unique. For this reason, we can use the martingale approach as in \cite[Chap.~20]{Bjork} to solve the problem \eqref{primal}, which uses the fact that every sufficiently regular payoff 
can be represented as an admissible portfolio via Equation \eqref{completeness}, with $\Delta \in {\mathcal A}$. However, here the assets follow an arithmetic dynamics (as e.g. in \cite{BPV,HinWag,Oksendal}), differently from the geometric dynamics of \cite{Bjork} (and of the vast majority of financial literature), i.e. where the assets' dynamics are proportional to asset prices themselves. While this does not change the structure of the optimal terminal wealth $X_T^*$, and consequently of the optimal wealth process $X^* = (X^*_t)_t$, it will have an impact on the optimal portfolio strategy $\Delta^*$. 

 As in \cite[Chap.~20]{Bjork}, we incorporate 
the budget constraint in Equation \eqref{admissible} into  
the Lagrange function
\begin{align*}
 L(X_T^{\Delta}, \lambda)&:= \bE_{\bP}[u(X_T^{\Delta})]-\lambda\left(\bE_{\bQ}[X_T^{\Delta}] - x \right)\\
 &= \bE_{\bP} \left[ u(X^{\Delta}_T)-\lambda\left(X_T^{\Delta} Z_T - x\right)\right],
\end{align*}
 which is used in the problem 
$$ \displaystyle \inf_{\lambda > 0} \sup_{\Delta \in \mathcal{A}} L(X^\Delta_T,\lambda) $$
%
Since in our setting the market is complete, this problem is brought back to the
formulation
$$ \displaystyle \inf_{\lambda > 0} \sup_{X_T} L(X_T,\lambda) $$
i.e., we now maximize in the generic random variable $X_T$, remembering that, once we obtain the maximizer $X_T^*$, in order to obtain the optimal portfolio strategy $\Delta^*$  we use a martingale representation for the optimal portfolio $X_T^* = X_T^{\Delta^*}$.  \\
For $\lambda>0$ fixed, we solve the equation 
$$u'(X_T^*)-\lambda \frac{d \bQ}{ d \bP}=0,$$
from which 
\cite[Proposition 20.3]{Bjork}, the optimal terminal wealth $X_T^*$ results in
\begin{equation}\label{Xt}
    X_T^{*} = I\left(\lambda^* Z_T\right),
\end{equation}
with $I=(u')^{-1}$,   and the Lagrange multiplier $\lambda^*$ is the one realizing the budget constraint $\bE_\bQ[X^*_T] = x$ in Equation \eqref{admissible}. 



\begin{remark}
We notice that the result for the optimal terminal wealth in Equation \eqref{Xt} is not new, and apparently also does not depend on the fact that we are using Gaussian Volterra processes in our model. This is because the result in Equation \eqref{Xt} is true for all complete markets with equivalent martingale measure $\bQ$ having density $Z_T$ with respect to the real-world probability $\bP$, where agents want to maximize a utility function $u$. However, the dependency on the specific features of our market, namely on the kernels $K_i$, already appears implicitly when we take into account that the density $Z_T$ has the Girsanov representation given by Equation \eqref{ZT}, which contains the kernels in $\theta$ through the representation results of Proposition 1. Moreover, when characterizing the optimal portfolio, its dependence on the particular structure of our market will be even more evident, see e.g. Proposition \ref{optimalportfolio} below.   
\end{remark}

\subsection{The case of CRRA utility functions} 

Let us now consider 
the case when the utility function is of Constant Relative Risk Aversion (CRRA) type, i.e. $u(x) := \frac1\gamma x^\gamma$, with $\gamma < 1$, $\gamma \neq 0$, or $u(x) = \log x$ (which is often seen conventionally as "the
case $\gamma = 0$", see e.g. the discussion in \cite[Chapter 20.7]{Bjork}). Then, in order for $\Delta$ to be an admissible strategy, we must impose that $X^\Delta_T \geq 0$ in the case $\gamma \in (0,1)$ and $X^\Delta_T > 0$ in all the other cases. Moreover, for all $\gamma < 1$ (including the case "$\gamma = 0$", which is the case of a log utility function) we have that $u'(x) = x^{\gamma - 1}$, thus $I(y) = y^{\frac{1}{\gamma-1}}$. \\

\textit{The optimal $\lambda^*$}.
From \eqref{Xt}, we have $X_T^*=\left( \lambda Z_T\right)^{\frac{1}{\gamma-1}}$,    and imposing $\bE_\bQ[X^*_T] = x$ we find  
%
$$\lambda^*=\left( \frac{x}{\bE_{\bP}\left[  \left(Z_T\right)^{\frac{\gamma}{\gamma-1}}  \right]}\right)^{\gamma-1}.$$
Hereafter, we will follow \cite[Ch. 20]{Bjork}. Let us set
$$ H_0=\bE_{\bP}\left[Z_T^{-\beta}\right], \qquad \mbox{ with } \qquad \beta=\frac{\gamma}{1-\gamma},$$
%
Then, the optimal terminal wealth 
writes
\begin{equation}\label{XTopt}
X_T^* = \left(\lambda^* Z_T \right)^{\frac{1}{\gamma-1}}
= \frac{x}{H_0}Z_T^{\frac{1}{\gamma-1}},
\end{equation}
and the optimal expected utility is
\begin{equation*}\label{optExpU}
\bE_\bP[u(X_T^*)] = \bE_\bP\left[\frac{(X_T^*)^{\gamma}}{\gamma}\right]= 
\frac{x^\gamma}{\gamma}H_0^{1-\gamma},
\end{equation*}
%


\textit{The Optimal Wealth Process.}
\noindent In \eqref{XTopt}, we have computed the optimal terminal wealth 
$X_T^*$, but it is also possible to find an explicit formula for the entire optimal wealth process $X^*$.  In this, we follow \cite[Chapter 20.5.2]{Bjork}, but for the reader's convenience here we write an {\em ad-hoc} derivation for our case. 

\begin{proposition}\label{OWP}
The optimal wealth process $X^*= (X^*_t)_{t\geq 0}$  is given by
\begin{equation}\label{optimalX}
X^*_t= x \frac{H_t}{H_0} Z_t^{-\frac{1}{1-\gamma}},
\end{equation}
where
\begin{equation}\label{Ht}
H_t:= \bE_\bP\left[ \left. \frac{Z_T^{-\beta}}{Z_t^{-\beta}} \right|{\mathcal F}_t \right]  =\exp\left\{\frac{1}{2}\int_t^T \frac{\beta}{1-\gamma}|\theta_s|^2 \; ds\right\}. 
\end{equation}
\end{proposition}
\blue{For the proof we refer to the Appendix.}

\textit{The Optimal Portfolio.}
Let us denote by   $\mu_H(t) := -\frac{1}{2} \frac{\beta}{1-\gamma}|\theta_t|^2 $  
the drift of the process $H$ defined in \eqref{Ht},  so that 
\begin{equation*}
  dH_t=H_t \mu_H(t) dt,   
\end{equation*}
  with the terminal condition $H_T = 1$.  


 \begin{proposition} \label{optimalportfolio}
The optimal portfolio 
process $\Delta^* = (\Delta^*_t)_{t\geq 0} \in {\mathcal A}$ is given by 
\begin{equation*} \label{Deltastar}
  \Delta^*_t= X^*_t \frac{1}{1-\gamma} \theta_t \cdot \blue{\bar K_{\mathrm{left}}(t)^{-1}}, \quad \text{for a.a.} \quad t \leq T
 \end{equation*}
 \end{proposition} 
 \blue{For the proof we refer to the Appendix.}

\blue{\section{Option pricing}\label{Sec:Option}}

 In electricity markets, we can price  the classical vanilla options, i.e.~calls and puts, written on the   
assets    
$F(\cdot,T_j)$, $j = 1,\ldots,n$. Assume, for simplicity  in notations, that the discount rate $r=0$, and that the market is arbitrage-free in the spirit of Proposition 1, i.e. that there exists a probability $\bQ \sim \bP$ such that forward prices $F(\cdot,T_j)$ are $\bQ$-martingales for all $j = 1,\ldots,n$. Then, for  example, the risk-neutral price  at time $t$ of a vanilla call option with strike price $K$ and maturity $T$ written on the forward contract $F(\cdot,T_j)$, with $T_j > T$, is given by
\begin{equation} \label{callonforward}
C(t,T) = \bE_\bQ[ 
(F(T,T_j) - K )^+\ |\ {\cal F}_t ], 
\end{equation}
while the price of a put option with the same strike price $K$ and maturity $T$ is given by
\begin{equation} \label{putonforward}
P(t,T) = \bE_\bQ[ 
(K - F(T,T_j) )^+\ |\ {\cal F}_t ]. 
\end{equation}

We can immediately verify that, as in any other arbitrage-free market, the call-put parity holds in the following form.

\begin{lemma} \label{CPparity}
If $C(t,T)$ and $P(t,T)$ denote respectively the call and put prices at time $t$, both with strike price $K$ and maturity $T$ written on the forward contract $F(\cdot,T_j)$, with $T_j > T$ as in Equations \eqref{callonforward}-\eqref{putonforward}, then we have
\begin{equation*} 
C(t,T) - P(t,T) = 
F(t,T_j) - K
\end{equation*}
\end{lemma}
\begin{proof}
From Equations \eqref{callonforward}-\eqref{putonforward}, it easily follows that 
\begin{eqnarray*}
C(t,T) - P(t,T) 
& = & \bE_\bQ[ 
[(F(T,T_j) - K)^+ - (K - F(T,T_j) )^+]\ |\ {\cal F}_t ] = \\
& = & 
\bE_\bQ[ F(T,T_j) - K\ |\ {\cal F}_t ] = 
F(t,T_j) - K 
\end{eqnarray*}
by the martingale property of $F(\cdot,T_j)$.
\end{proof}

By using the Bachelier formula adapted to this context (see e.g. \cite[Eq. (3)]{CKTW}), we can also produce a closed formula for the price of call options on forwards (and consequently, by the call-put parity above, also for the price of put options). 

\begin{proposition} \label{callprice}
The price of a call option at time $t$ written on the forward contract $F(\cdot,T_j)$ with strike price $K$ and maturity $T$ with $T_j > T$ is 
\begin{eqnarray*}
C(t,T) & = & 
\sigma(t,T,T_j) (d_j N(d_j) + n(d_j)) = \\
& = & 
(F(t,T_j) - K) N(d_j) + \sigma(t,T,T_j) n(d_j)
\end{eqnarray*}
where
\begin{equation} \label{sigma3}
\sigma(t,T,T_j) :=   \sqrt{\int_t^T |K(T_j,u)|^2\ du}, \qquad d_j(t,T,T_j) := \frac{F(t,T_j) - K}{\sigma(t,T,T_j)}, 
\end{equation}
and the functions $n$ and $N$ are respectively the density and the cumulative distribution functions of a $N(0,1)$ law, i.e.
$$ n(x) := \frac{1}{\sqrt{2 \pi}} e^{- \frac12 x^2}, \qquad N(y) := \int_{-\infty}^y n(x)\ dx. $$
Moreover, the hedging portfolio for this call option is composed by the only asset $F(\cdot,T_j)$ ( together with the money market account), and the quantity of this asset to be held at time $t$ is $N(d_j(t,T,T_j))$.
\end{proposition}

\begin{proof}
It is sufficient to notice that, under the risk-neutral probability $\bQ$, the distribution of $F(T,T_j)$ is Gaussian with mean $F(t,T_j)$ and variance $\sigma^2(t,T,T_j)$,  with $\sigma(t,T,T_j)$ given by Equation \eqref{sigma3}. The conclusion about the call price follows then from \cite[Eq. (3)]{CKTW}, and the hedging portfolio from \cite[Section 5.1]{CKTW}.
\end{proof}

 \begin{remark} 
Notice that in Proposition \ref{callprice} we did not need to assume completeness. In fact, even if the market is not complete, the forward price $F(\cdot,T_j)$ will have distribution $N(F(t,T_j),\sigma^2(t,T,T_j))$ regardless on the particular equivalent martingale measure $\bQ$ governing prices. This is different e.g. from stochastic volatility models, where different equivalent measures $\bQ$ would give different distributions to the volatility process, thus the call price would depend on which particular distribution the volatility has under the particular $\bQ$ chosen. As a consequence, in our model call prices are invariant with respect to $\bQ$, as the pricing formulas in Proposition \ref{callprice} do not depend on $\bQ$ but ultimately only on the kernel vector $K(T_j,\cdot)$, which does not vary with $\bQ$ (and of course on $F(t,T_j)$ and $K$). Other more general derivatives, instead, could have a no-arbitrage price possibly dependent on $\bQ$. 
\end{remark}

 \begin{remark} 
The results above are true for discounted prices, i.e. for a situation where the interest rate applied to prices is $r \equiv 0$. When one wants to derive explicitly the same results taking explicitly into account the presence of a risk-free short rate $r(u)$, $u \in [t,T]$, it is not difficult to see that in the case when $r$ is deterministic the call-put parity formula modifies into
$$ C(t,T) - P(t,T) = e^{-\int_t^T r(u)\ du} (
F(t,T_j) - K) $$
and the call price into 
$$ C(t,T) = e^{-\int_t^T r(u)\ du}  \sigma(t,T,T_j) (d_j(t,T,T_j) N(d_j(t,T,T_j)) + n(d_j)(t,T,T_j)). $$
%
The more general case when the intensity $(r_u)_u$ is a stochastic process, possibly dependent on $F(\cdot,T_j)$, can be treated with the usual change-of-numeraire techniques, see e.g. \cite[Chapter 26]{Bjork}. Here we decided not to present this in detail so as not to further burden the reader. 
\end{remark}

However, in electricity markets other illiquid products can be found as well, usually more involved than the vanilla products seen above. A notable example is the so-called Reliability Option, which is present in several national markets (see e.g. \cite{AFFV}), and which has the peculiarity that its payoff is defined on the spot price $S$ over the time span $[T_1,T_2]$, with $0 \leq T_1 < T_2 \leq \bar T$. 
More in detail, the payoff of a Reliability Option with fixed strike price $K$ written on the time span $[T_1,T_2]$ is found to be equal to \cite{AFFV}
\begin{eqnarray} 
RO(t,T_1,T_2) & = & \bE_\bQ\left[\left. \int_{T_1}^{T_2} 
(S_T - K)^+ dT\ \right| {\mathcal F}_t \right] = \nonumber \\
& = & \int_{T_1}^{T_2} 
\bE_\bQ\left[ (S_T - K)^+\ |\ {\cal F}_t  \right] dT. \label{RO}
\end{eqnarray}
\begin{proposition}
The price at time $t = 0$ of the Reliability Option as defined in Equation \eqref{RO} is equal to
$$ RO(0,T_1,T_2) = \int_{T_1}^{T_2} 
\sigma(0,T,T) (d_T N(d_T) + n(d_T))\ dT $$
with $\sigma(0,T,T)$ as defined in Equation \eqref{sigma3}, 
$$ d_T := \frac{\varphi_\bQ(T) - K}{\sigma(0,T,T)} $$
and $\varphi_\bQ$ deterministic seasonality of $S$ as defined in Equation \eqref{newphifinal}.
\end{proposition}
\begin{proof}
The proof follows from the fact that, under $\bQ$, we have 
that $S_T \sim N(\varphi_\bQ(T), \sigma^2(0,T,T) )$. It is then sufficient to use this in the right-hand side of Equation \eqref{RO} with the Bachelier formula already used in Proposition \ref{callprice}. 
\end{proof}

\section{Conclusions} 

We introduce a non-Markovian model for the spot price of electricity, based on a $n$-factor Gaussian Volterra process, obtained by stochastic integrals of Volterra kernels. As is customary in electricity markets, we assume that the spot price is not traded, but forward contracts with maturities $T_1 < T_2 < \ldots < T_m$ are traded in the market. We discuss conditions for absence of arbitrage and completeness. We characterize the dynamics of forward prices when there is no arbitrage in terms of an equivalent martingale measure, finding that in that case the risk-neutral dynamics of forward prices has zero drift and diffusion coefficients equal to the Volterra kernels with a variable frozen to the forward maturity. Moreover, we find that completeness is equivalent to the invertibility (in a generalized sense) of the matrix formed by the diffusion coefficients of the traded forwards.  In all the above discussion, we modeled forward prices as having instantaneous delivery, while real contracts in electricity markets deliver electricity during a time span $[T_1,T_2]$: however, this is a common market practice, which we show to be justified in our framework, as the tracking error between the real asset and the one that we model goes to zero as a power of the period length $T_2 - T_1$, thus this approximation is particularly robust for hourly or daily contracts. \blue{However, when the delivery period is longer, e.g of the order of one month, one trimester or one year, in which case the approximation of Lemma \ref{2H} can possibly be not accurate enough, we can still use the results of Sections \ref{aac} and \ref{sec2.2}, but take care of substituting the kernels {$K(T_j,t)$}, relative to ``instantaneous'' forwards introduced in Equation \eqref{S-F}, with the kernels $\bar K(t,T_j,T_k)$ in Equation \eqref{barK}, relative to the forward contracts which are actually traded in the market.} \\
\blue{After this, we analyze in detail examples of the most used Gaussian Volterra processes in our particular framework, especially concerning market completeness. The case when the Volterra kernels have  functional forms which are different among each other is usually quite straightforward, so we just give one example of this case in Section \ref{mixed}. Instead, the more interesting case is when the Volterra kernels refer to the same kind of processes. Thanks to a novel representation of Ornstein-Uhlenbeck (OU) processes driven by Gaussian Volterra processes, we show how to treat fractional Brownian motions as a particular case of fractional OU processes. We then pass to analyze particular instances of fractional processes, starting from Riemann-Liouville (RL) processes: in this case, we find that $n$ independent RL processes generate a complete market if their Hurst exponents are all different from each other. The case of RL-driven OU processes is more involved, and we are able to prove that two such processes generate a complete market in the case when the mean-reversion speed is the same and the Hurst exponents are different. We use the same framework, i.e. two OU processes with same mean-reversion speed, also in the case of fractional Brownian motions  with Hurst exponents both greater than 1/2;  
 we prove the market completeness in a quite straightforward way, as the diffusion matrix is invertible for all times.  Instead, we do not present the case of OU processes driven by fBm with Hurst exponent smaller than 1/2. We only consider two fBm driving the spot prices, with different Hurst exponents smaller than 1/2, we find that the diffusion matrix is invertible at all times but one, but the assumptions for completeness are still satisfied.\\ 
 Finally, we find the optimal investment in forward markets, analyzing in detail the particular case of CRRA utility functions, for which we characterize explicitly the optimal portfolio strategy. We also present option pricing formulas: in the particular case of vanilla calls and puts written on a traded forward, we find that the price is unique also in incomplete markets and given by a version of the Bachelier formula, and we also give the hedging strategy. However, in order to have a unique price for other kind of derivatives, possibly more exotic, completeness should be assumed: as an example of this, we give the pricing formula for a particular option in electricity market, called reliability option.}\\ 
\blue{We underline that if we substitute the kernels $K(T_j,t)$, relative to ``instantaneous'' forwards introduced in Equation \eqref{S-F}, with the kernels $\bar K(t,T_j,T_k)$ in Equation \eqref{barK}, 
one can carry out the same kind of computation that we show in Sections \ref{GVP}, \ref{Sec:Optimal}, and \ref{Sec:Option} for the specific situation.}
\appendix

\section{Appendix}\label{appendix}

 \emph{Proof of Proposition \ref{proposition1}.} 

\noindent  The process defined in Equation \eqref{WQ} is a $\bQ$-Brownian motion according to the Girsanov theorem \cite{Bjork}, and Equations \eqref{newphi2} and \eqref{phiQ} follow from this and from Equation \eqref{St}.  

 Assume now that there is no arbitrage, i.e. Equation \eqref{S-F} is true for all $j = 1,\ldots,m$. This, together with Equation \eqref{newphi2}, implies  
\begin{eqnarray*}
\begin{split}
   F(t,T_j) = & \bE_{\bQ}\left[ \left. 
   \int_0^{T_j} K(T_j,s) \cdot dW_s^{\bQ} + \varphi_\bQ(T_j) \right| \cF_t \right] \\
   & = 
   \int_0^t K(T_j,s) \cdot dW_s^{\bQ} + \varphi_\bQ(T_j).
   \end{split}
\end{eqnarray*}
whence  the dynamics of $\{F(t,T_j), t \in [0,T_j] \}$ under $\bQ$ is given by 
Equation \eqref{FQ}, with $F(\cdot,T_j)$ resulting in a $\bQ$-martingale such that $F(T_j,T_j) = S_{T_j}$. By substituting the definition of $W^\bQ$, given by Equation \eqref{WQ}, in the dynamics \eqref{FQ}, we obtain
%
\begin{equation}\label{dFP} 
dF(t,T_j) = K(T_j,t) (dW_t + \theta_t dt) = K(T_j,t) \theta_t dt + K(T_j,t) dW_t,
\end{equation}
which, compared with Equation \eqref{FP}, gives the desired conclusion.  

  Conversely, now assume that Equations \eqref{FQ} and \eqref{musigma} hold, which consistently give the dynamics in Equation \eqref{FP}. Then, combining  Equation \eqref{FQ} with its terminal condition and Equation \eqref{newphi2}, we obtain  
$$ F(t,T_j) = S_{T_j} - \int_t^{T_j} K(T_j,s) \cdot dW_s^\bQ = \varphi_\bQ(T_j) + \int_0^t K(T_j,s) \cdot dW_s^\bQ. $$
Since $\varphi_\bQ(T_j)$ is deterministic and $\int_0^t |K(T_j,s)|^2\ ds < + \infty$ for all $t \leq T_j < \bar T$, it is clear that $F(\cdot,T_j)$ are $\bQ$-martingales for all $j = 1,\ldots,m$. 

\vspace{2ex}

\noindent \emph{Proof of Theorem \ref{Kmatrix}.}

\noindent We proceed by induction. Obviously, the statement is true for $n=1$. Let us assume that it is true for $n-1$. 
We can transform our determinant $\Delta$ as follows:

$$\Delta(x_1, \ldots, x_n)=\prod_{i=1}^n x_i ^{\alpha_1} \tilde \Delta (x_1, \ldots, x_n),$$
where 

$$\tilde \Delta_n(x_1, \ldots, x_n)= 
\begin{vmatrix}
1 & 1 & \ldots & 1\\ 
x_1^{\alpha_2-\alpha_1 } & x_2^{\alpha_2-\alpha_1 } & \ldots & x_n^{\alpha_2-\alpha_1} \\
\vdots & \vdots & \vdots & \vdots \\ 
x_1^{\alpha_n-\alpha_1} & x_2^{\alpha_n-\alpha_1} & \ldots & x_n^{\alpha_n-\alpha_1}
\end{vmatrix}.$$
Denote $\beta_i= \alpha_i -\alpha_1$, $i=1, \ldots, n$, and transform $\tilde \Delta_n$ into
\begin{align*}
    \tilde \Delta_n(x_1, \ldots, x_n) &= 
\begin{vmatrix}
1 & 0 & \ldots & 0\\ 
x_1^{\beta_2} & x_2^{\beta_2}-x_1^{\beta_2} & \ldots & x_n^{\beta_2}-x_1^{\beta_2} \\
\vdots & \vdots & \vdots & \vdots \\ 
x_1^{\beta_n} & x_2^{\beta_n}-x_1^{\beta_n} & \ldots & x_n^{\beta_n}-x_1^{\beta_n}
\end{vmatrix}
\\ \notag
&& \\ 
&= 
\begin{vmatrix}
 x_2^{\beta_2}-x_1^{\beta_2} & \ldots & x_n^{\beta_2}-x_1^{\beta_2} \\
\vdots & \vdots & \vdots \\ 
  x_2^{\beta_n}-x_1^{\beta_n} & \ldots & x_n^{\beta_n}-x_1^{\beta_n}
\end{vmatrix}.
\end{align*}
Since we have assumed that the statement holds for $n-1$ 
we can assert that

\begin{align*}
\frac{\partial^{n-1}\tilde \Delta_n(x_1, \ldots, x_n)}{\partial x_2 \ldots \partial x_n}&= 
\begin{vmatrix}
\beta_2 x_2^{\beta_2-1} & \ldots & {\beta_2}x_n^{\beta_2-1} \\
\vdots & \vdots & \vdots  \\ 
\beta_n x_2^{\beta_n-1} & \ldots & x_n^{\beta_n-1}
\end{vmatrix}
\\ \notag
&& \\ 
&= \frac{\prod_{i=2}^n \beta_i}{\prod_{i=2}^n x_i}
\begin{vmatrix}
x_2^{\beta_2} & \ldots & x_n^{\beta_2} \\
\vdots & \vdots & \vdots  \\ 
x_2^{\beta_n} & \ldots & x_n^{\beta_n}
\end{vmatrix}>0.
\end{align*}
Now, let us consider subsequently the derivative, starting with 
$$ 0 <\frac{\partial^{n-1}\tilde \Delta_n(x_1, \ldots, x_n)}{\partial x_2 \ldots \partial x_n}= \frac{\partial}{\partial x_n}\left(\frac{\partial^{n-2} \tilde \Delta_n(x_1, \ldots, x_n)}{\partial x_2 \ldots \partial x_{n-1}}\right).$$
We see that $$\frac{\partial^{n-2} \tilde \Delta_n(x_1, \ldots, x_n)}{\partial x_2 \ldots \partial x_{n-1}}$$ is increasing in $x_n$, but 

$$ \frac{\partial^{n-2} \tilde \Delta_n(x_1, \ldots, x_n)}{\partial x_2 \ldots \partial x_{n-1}}=
\begin{vmatrix}
\beta_2 x_2^{\beta_2-1} & \ldots & {\beta_2}x_{n-1}^{\beta_2-1} & x_n^{\beta_2}-x_1^{\beta_2} \\
\vdots & \vdots & \vdots & \vdots  \\ 
\beta_n x_2^{\beta_n-1} & \ldots & x_{n-1}^{\beta_n-1} & x_n^{\beta_n}-x_1^{\beta_n}
\end{vmatrix},$$
and it equals zero if $x_n=x_1$.
Therefore, $$\frac{\partial^{n-2} \tilde \Delta_n(x_1, \ldots, x_n)}{\partial x_2 \ldots \partial x_{n-1}}>0$$ for $x_n>x_1$, that is our case.
Thus, we can proceed in the same way and get that 
$$\frac{\partial \tilde \Delta(x_1, \ldots, x_n)}{\partial x_2}>0,$$ for $x_3>x_1$, that is our case.
But $\tilde \Delta_n(x_1, \ldots, x_n)=0$  if $x_2=x_1$, therefore $ \tilde \Delta(x_1, \ldots, x_n)>0$.


\vspace{2ex}
\noindent \emph{Proof of Theorem \ref{OWP}.}
Since the wealth process 
of a self-financing portfolio is a $\bQ$-martingale, we have 
\begin{equation*}
X_t^*= \bE^{\bQ}[X_T^*|\cF_t]=\frac{x}{H_0}\bE^{\bQ}\left[Z_T^{-\frac{1}{1-\gamma}}|{\cF_t}\right].   
  \end{equation*}
By the abstract Bayes' formula, we obtain
\begin{equation}\label{Xt*}
X_t^*= \frac{x}{H_0} \frac{\bE^{\bP} \left[ Z_T^{-\beta}|\cF_t\right]}{Z_t}.
\end{equation}
 From \eqref{ZT} we obtain 
$$ Z_T^{-\beta}=\exp \left\{\int_0^T \beta \theta_s \cdot \;dW_s +\frac{1}{2} \int_0^T  \beta |\theta_s|^2\; ds \right\},$$
that looks almost like a Radon-Nikodym derivative: this leads us to define the $P$-martingale $Z^0$ as 
$$ Z^0_t= \exp \left\{\int_0^t \beta \theta_s \cdot \;dW_s - \frac{1}{2} \int_0^t  \beta^2 |\theta_s|^2 \; ds \right\},$$
whose dynamics is 
$$dZ^0_t=Z^0_t \beta \theta_t \cdot \;dW_t.$$
Thus, now we have that
\begin{equation}\label{Zb}
Z_T^{-\beta}=Z^0_T \exp \left\{ \frac{1}{2}\int_0^T \frac{\beta}{1-\gamma} |\theta_t|^2 \; dt \right\},
\end{equation}
 Now, from \eqref{Zb} 
we have
\begin{align*}
   \bE^{\bP}\left[ Z_T^{-\beta} |\cF_t\right]&= \bE^{\bP}\left[ Z_T^0 e^{\frac{1}{2}\int_0^T \frac{\beta}{1-\gamma} |\theta_s|^2 \; ds}| \cF_t\right]\\
   &=Z_t^0 e^{\frac{1}{2}\int_0^t \frac{\beta}{1-\gamma} |\theta_s|^2 ds} e^{\frac{1}{2}\int_t^T \frac{\beta}{1-\gamma} |\theta_s|^2 ds}\\
   &= Z_t^{-\beta}   e^{\frac{1}{2}\int_t^T \frac{\beta}{1-\gamma} |\theta_s|^2 ds}.
\end{align*}
  From Equation \eqref{Ht}, we have that
$$ H_t = \bE_\bP\left[ \left. e^{\int_t^T \beta \theta_s \cdot dW_s - \frac12 \int_t^T \beta |\theta_s|^2 ds} \right| {\mathcal F}_t \right] $$
Since $\theta$ is deterministic, we have that $\int_t^T \beta \theta_s \cdot dW_s$ is Gaussian and independent of ${\mathcal F}_t$, thus
$$ H_t = \exp\left( \frac12 \beta (1 - \beta) \int_t^T |\theta_s|^2 ds \right) $$
and Equation \eqref{Ht} follows from $1 - \beta = \frac{1}{1-\gamma}$. 
Putting all of this in Equation \eqref{Xt*} we obtain
$$X_t^*= \frac{x}{H_0} Z_t^{-1} Z_t^{-\beta} H_t=  x\frac{H_t}{H_0} Z_t^{-\frac{1}{1-\gamma}}.$$
%

\vspace{2ex}
\noindent \emph{Proof of Theorem \ref{optimalportfolio}.}
Starting from \eqref{optimalX} for the optimal wealth process, by the It\^o formula we have that 
\begin{equation*}
    d X_t^*= x Z_t^c \frac{H_t}{H_0} \left(\mu_H(t)  + \frac{1}{2} c(c-1)|\theta_t|^2\right) dt - x Z_t^c \frac{H_t}{H_0}  c \theta_t \cdot dW_t,
\end{equation*}
where $c=-(1-\gamma)^{-1}$. 
We then have
\begin{eqnarray}\label{dXPstar}
dX_t^* & = & X_t^* (- c \theta_t \cdot dW_t - c |\theta_t|^2 dt) = - c X_t^*\theta_t \cdot dW^\bQ_t 
\end{eqnarray}
By comparing the diffusion part in \eqref{dX} and \eqref{dXPstar}, we have
\begin{equation}\label{sysDK}
   \Delta_t \bar K(t)=-c X_t^*\theta_t.
\end{equation}
System \eqref{sysDK} admits solution (for a.a. $t \leq T < T_1$) given that by Theorem \ref{complete1} there exists a left inverse of $\bar K(t)$ for a.a. $t \leq T_1$.
Thus, taking
$$\Delta_t=-c X_t^* \theta_t \cdot \blue{\bar K_{\mathrm{left}}(t)^{-1}},$$
we obtain \eqref{sysDK}; this solution is our $\Delta_t^*$.
%
%
Since $X^*$ appears linearly in $\Delta^*$, multiplied by a square integrable deterministic function of time, we also have that Equation \eqref{dX} has a unique strong solution, thus $\Delta^*$ is admissible and optimal.

\end{document}